\documentclass[11pt]{article}
\usepackage{a4wide,amsmath,amsbsy,amsfonts,amssymb, 
stmaryrd,amsthm,mathrsfs,graphicx, amscd, tikz-cd, cancel, supertabular}
\usepackage[normalem]{ulem}
\usetikzlibrary{matrix,arrows,decorations.pathmorphing}

\usepackage{xr}

\usepackage{subcaption}

\usepackage{tikz}
\usetikzlibrary{matrix,arrows,decorations.pathmorphing}
\usepackage{qtree}

\usepackage[top=1.2in,bottom=1.2in,left=1.21in,right=1.21in]{geometry}
\usepackage[
	hypertexnames=false,
	hyperindex,
	pagebackref,
	pdftex,
	breaklinks=true,
	bookmarks=false,
	colorlinks,
	linkcolor=blue,
	citecolor=red,
	urlcolor=red,
]{hyperref}
\usepackage{hyperref}

\newcommand\bs{\backslash}

\newtheorem{theorem}{Theorem}[section]
\newtheorem{proposition}[theorem]{Proposition}

\newtheorem{lemma}[theorem]{Lemma}

\theoremstyle{definition}

\newtheorem{definition}[theorem]{Definition}

\numberwithin{equation}{section}

\theoremstyle{definition}
\newtheorem{remark}[theorem]{Remark}

\newcommand{\As}{\mathscr{A}}

\newcommand{\Cs}{\mathscr{C}}

\newcommand{\calM}{\mathcal{M}}

\newcommand{\pbf}{\mathbf{p}}
\newcommand{\qbf}{\mathbf{q}}
\newcommand{\gbf}{\mathbf{g}}

\newcommand{\Bc}{\mathcal{B}}

\newcommand{\Oc}{\mathcal{O}}

\newcommand{\half}{\tfrac{1}{2}}

\newcommand{\Bb}{\mathbb{B}}
\newcommand{\Cb}{\mathbb{C}}
\newcommand{\Db}{\mathbb{D}}\newcommand{\C}{\mathbb{C}}

\newcommand{\Fb}{\mathbb{F}}
\newcommand{\F}{\mathbb{F}}
\newcommand{\G}{\Gamma}

\newcommand{\Mb}{\mathbb{M}}

\newcommand{\Pb}{\mathbb{P}}
\newcommand{\Qb}{\mathbb{Q}}

\newcommand{\Rb}{\mathbb{R}}

\newcommand{\Zb}{\mathbb{Z}}
\newcommand{\Z}{\mathbb{Z}}

\newcommand{\Sf}{\mathfrak{S}}
\newcommand{\Af}{\mathfrak{A}}

\newcommand{\beq}{\begin{eqnarray}}
\newcommand{\eeq}{\end{eqnarray}}

\newcommand\ssm{\smallsetminus}

\newcommand{\ir}{{ir}}

\newcommand{\ev}{\mathrm{ev}}

\newcommand{\Hf}{\mathfrak{H}}

\DeclareMathOperator{\Aut}{Aut}

\DeclareMathOperator{\spin}{Spin}
\DeclareMathOperator{\PGL}{PGL}
\DeclareMathOperator{\GL}{GL}

\DeclareMathOperator{\In}{Inn}

\DeclareMathOperator{\Sym}{Sym}

\DeclareMathOperator{\End}{End}

\DeclareMathOperator{\Mod}{Mod}

\DeclareMathOperator{\Orth}{O}
\DeclareMathOperator{\Out}{Out}

\DeclareMathOperator{\SO}{SO}

\DeclareMathOperator{\la}{\langle}
\DeclareMathOperator{\ra}{\rangle}

\DeclareMathOperator{\PSL}{PSL}
\DeclareMathOperator{\SL}{SL}

\title{Geometry of the Wiman-Edge pencil and the Wiman curve}

\author{Benson Farb and Eduard Looijenga \thanks{The first author was supported in part by National Science Foundation Grant Nos. DMS-1105643 and DMS-1406209. The second author is supported by the Chinese National Science Foundation. Both authors are supported by the Jump Trading Mathlab Research Fund. }}

\begin{document}
\maketitle
\begin{abstract}
The {\em Wiman-Edge pencil} is the universal family $C_t, t\in\mathcal B$ of projective, genus $6$, complex-algebraic curves admitting a faithful action of the icosahedral group $\Af_5$.  The curve $C_0$, discovered by Wiman in 1895 \cite{Wiman} and called the 
{\em Wiman curve}, is the unique smooth, genus $6$ curve admitting a faithful action of the symmetric group $\Sf_5$.  In this paper we give an explicit uniformization of $\mathcal B$ as a non-congruence quotient $\Gamma\backslash \Hf$ of the hyperbolic plane $\Hf$, where $\Gamma<\PSL_2(\Z)$ is a subgroup of index $18$. We also give modular interpretations for various aspects of this uniformization, for example for the degenerations of $C_t$ into $10$ lines (resp.\ $5$ conics) whose intersection graph is the Petersen graph (resp.\ $K_5$).  

In the second half of this paper we give an explicit arithmetic uniformization of the Wiman curve $C_0$ itself as the quotient $\Lambda\backslash \Hf$, where $\Lambda$ is a principal level $5$ subgroup of a certain ``unit spinor norm'' group of M\"{o}bius transformations.  We then prove that $C_0$ is a certain moduli space of Hodge structures, endowing it with the structure of a Shimura curve of indefinite quaternionic type.  
\end{abstract}

\tableofcontents

\section{Introduction}
In 1895 Wiman \cite{Wiman} discovered a smooth, projective curve $C_0$ of genus $6$ with automorphism group the symmetric group $\Sf_5$; indeed this is the unique such curve (see Theorem 3.3 of \cite{DFL} or Theorem~\ref{theorem:moduli0} below).   In 1981 Edge \cite{Edge} placed $C_0$ in a family 
$C_t, t\in{\mathcal B}$ (now called the {\em Wiman-Edge pencil}) of smooth, projective curves $C_t$ of genus $6$ with a faithful action of the icosahedral group $\Af_5$; it is universal among all such families (see Theorem 3.3 of \cite{DFL} or Theorem~\ref{theorem:moduli2}  below).   As explained by Edge \cite{Edge} (see \cite{DFL} for a more modern treatment), the family $C_t$ of curves appears naturally on a quintic del Pezzo surface $S$, with the Wiman curve   ``a uniquely special canonical curve of genus $6$'' on $S$:  the standard action of $\Sf_5$ on $S$ leaves $C_t$ invariant and leaves invariant exactly the curve $C_0$.  The base $\mathcal B$ of the Wiman pencil appears also as the moduli space of K3 surfaces with (a certain) faithful $\mu_2\times\Af_5$ action; see \S 5.3 of \cite{FL}.  For a number of recent papers on the Wiman-Edge pencil, see \cite{Cheltsov, CKS, DFL, FL, Za}. 

 The problem of finding uniformizations of moduli spaces is a classical one, but it is typically a difficult task.
The first main result of this paper is to give an explicit uniformization of the smooth locus of the base of the Wiman pencil. We also give modular interpretations for various aspects of this uniformization, for example for the degenerations of the family.

\begin{theorem}[{\bf Uniformization of the universal icosahedral family}]\label{theorem:moduli2} 
There exists a torsion-free subgroup $\Gamma\subset \PSL_2(\Z)$ of index $18$  such that $\Bc^\circ:=\Gamma\bs \Hf$ underlies  the base of a universal family $\Cs_{\Bc^\circ}\to\Bc^\circ$ of compact Riemann surfaces of genus $6$ endowed with a faithful $\Af_5$-action such that:
\begin{enumerate}
\item [(i)] The map which assigns to a member of $\Cs_{\Bc^\circ}\to\Bc^\circ$ its $\Af_5$-quotient is represented by the 
natural map $\Gamma\bs \Hf\to \PSL_2(\Z)\bs \Hf$. 

\medskip
\item [(ii)] The natural completion $\Bc$ of $\Bc^\circ$ is of genus zero and has five cusps, one of width $2$,  two of width $3$ and  two of  width $5$.

\medskip
\item [(iii)]  The family $\Cs_{\Bc^\circ}\to\Bc^\circ$ extends naturally to a Deligne-Mumford stable family $\Cs_{\Bc} \to \Bc$ of $\Af_5$-curves
such that the fiber over each cusp  has (after normalization) only rational irreducible components. It is one of
the following types, labeled  by its  dual intersection graph (see Figure \ref{figure:degenerations}):
\medskip

\begin{figure}[h]
\begin{subfigure}{.5\textwidth}\hspace{.8in}
\includegraphics[scale=0.40]{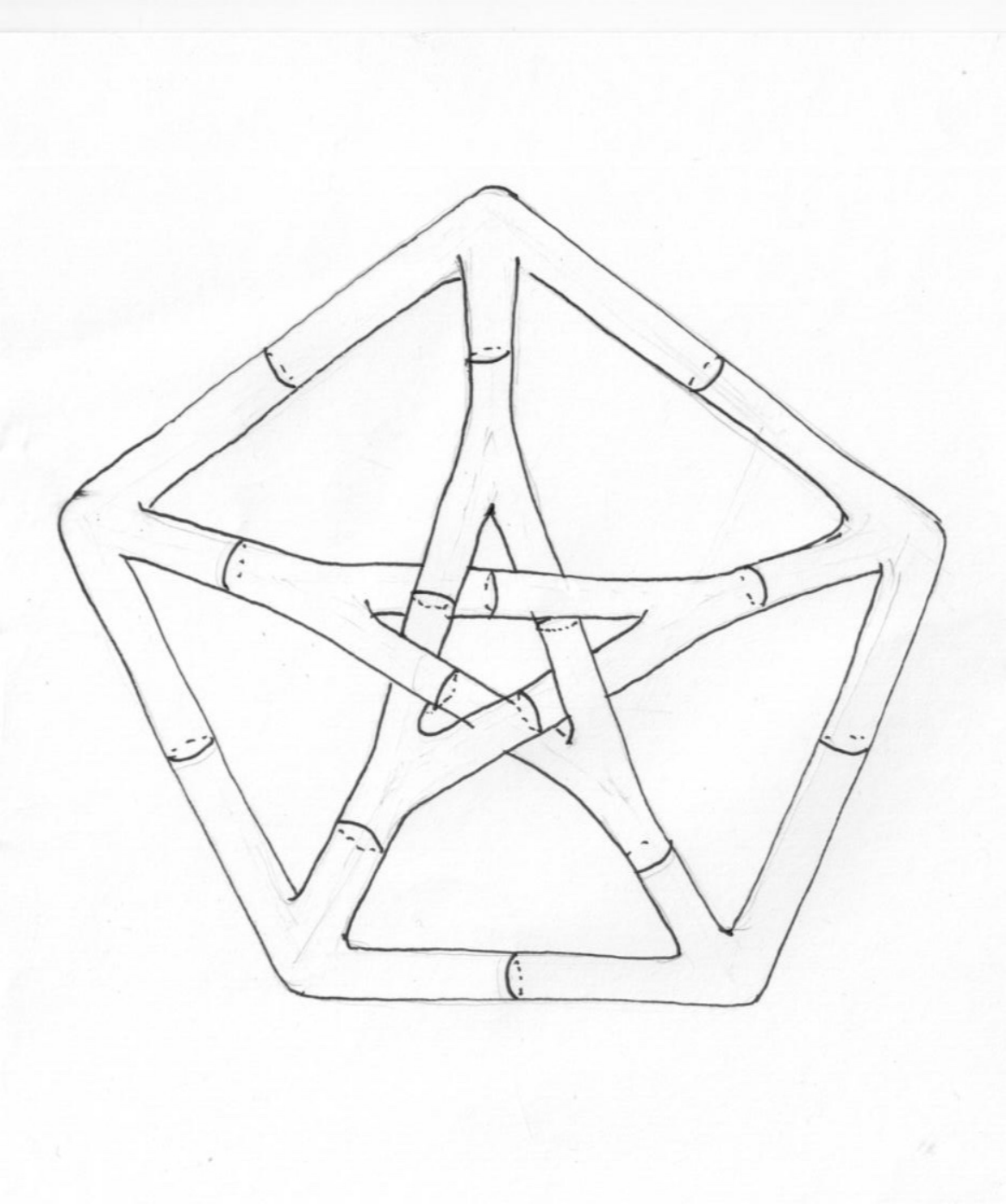}
\end{subfigure}
\begin{subfigure}{.5\textwidth}
\includegraphics[scale=0.45]{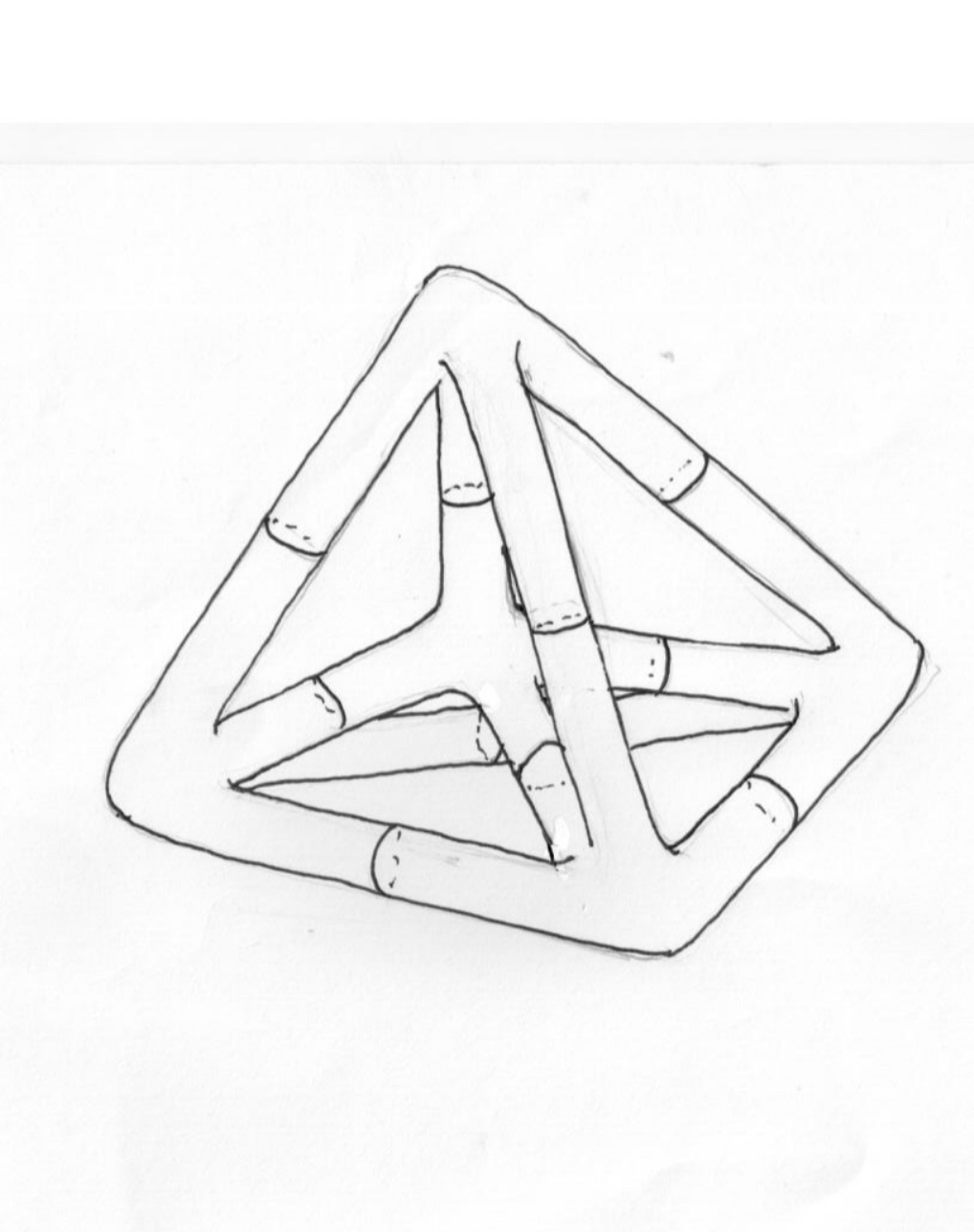}
\end{subfigure}
\caption{\footnotesize Two  genus $6$ surfaces with $\Af_5$-symmetry. The embedded circles are the vanishing cycles for the Petersen degeneration (cusp width $2$)  resp.\  the $K_5$ degeneration (cusp width $3$).}
\label{figure:degenerations}
\end{figure}

\begin{description}
\item[Petersen curve:] The Petersen graph (cusp width 2). 
\smallskip

\item[$K_5$ curve:] The complete graph on 5 vertices (cusp width 3).
\smallskip

\item[Dodecahedral curve:] The graph with one vertex and six loops attached (cusp width 5).
\end{description}
\medskip

\medskip
\item [(iv)] 
The involution $\iota$ of $\Gamma\subset \PSL_2(\Z)$ coming from the unique nontrivial outer automorphism of $\Af_5$ is defined by conjugation with an element of order two in $\PSL_2(\Z)$ which normalizes $\Gamma$; it exchanges the  two cusps of the same width  (see  Figure \ref{fig:baseWEpencil}).
\end{enumerate}
\end{theorem}

As we explain in Remark~\ref{rem:noncongruence} below, the group $\Gamma$ in Theorem~\ref{theorem:moduli2} is not a congruence subgroup.  We prove Theorem~\ref{theorem:moduli2} in \S\ref{sect:hypgeom}.  

Our next main result is the content of \S\ref{sect:combmodel}, where we find explicit affine, piecewise-Euclidean and hyperbolic metrics on each member of the family $C_t$.   These metrics allow us to view the $5$ degenerations into singular members of the family. One example of a piecewise-affine metric on a member of $C_t$ is given in Figure~\ref{figure:Dodec1}.

\begin{figure}
\centerline{\qquad\qquad\includegraphics[scale=0.3]{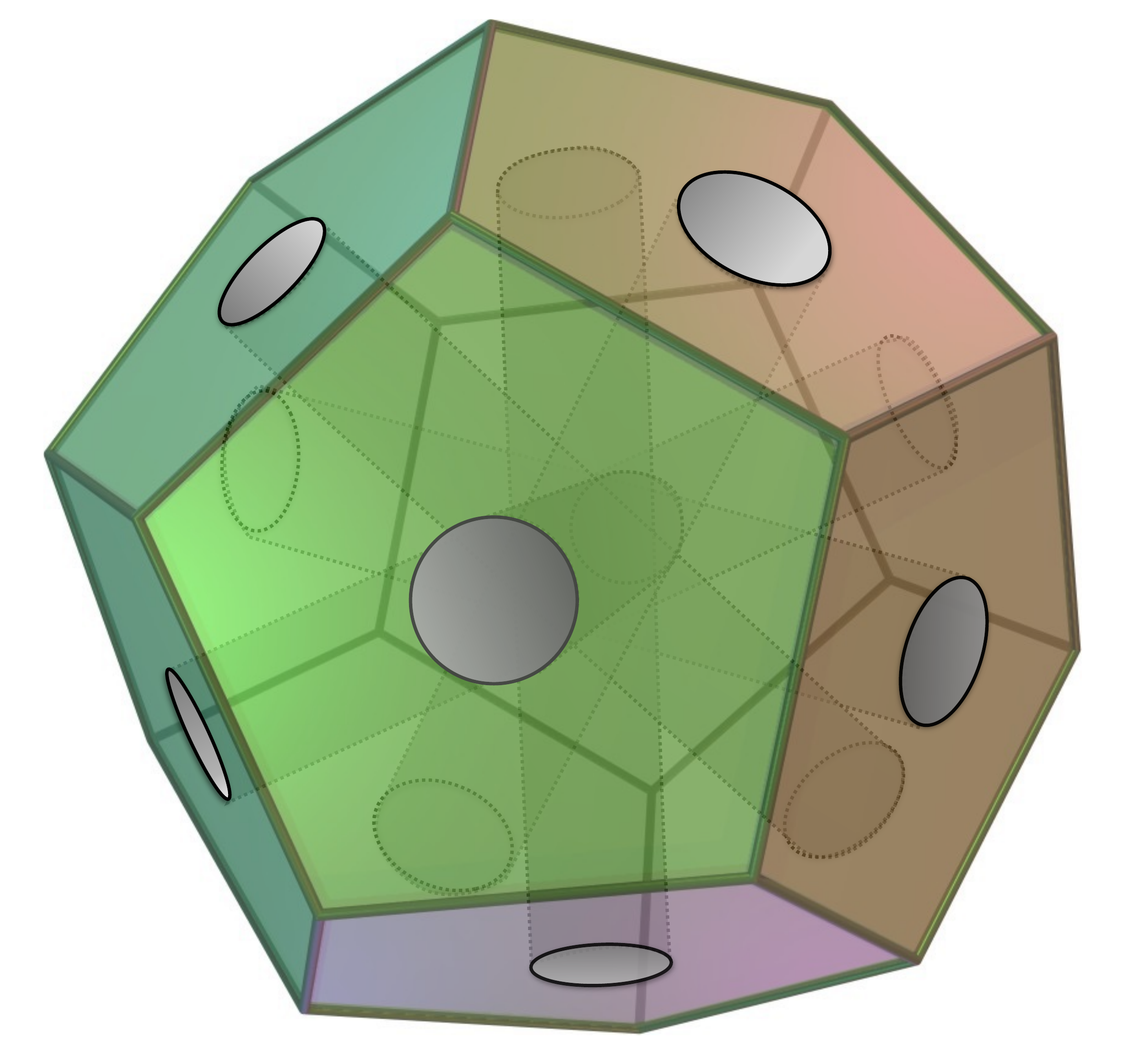}}
\caption{\small  A genus $6$ surface with $\Af_5$-action obtained as 
a dodecahedron with interior disks removed and opposite  boundary components  identified  (as suggested by the cylinders). 
The piecewise-linear metric determines a  conformal structure without singularities. The six antipodal pairs of boundary components
define as many vanishing cycles of a degeneration  into an irreducible  stable $\Af_5$-surface with an irreducible singular fiber (cusp width $5$).}\label{figure:Dodec1}
\end{figure}

The second half of this paper concerns the Wiman curve $C_0$ itself; again, this is the unique smooth, genus $6$ curve with faithful $\Sf_5$ action.  In \S\ref{section:triangle} we determine $C_0$ as a hyperbolic triangle group, and we prove that the Wiman curve is actually an {\em arithmetic curve}, that is, the quotient of the hyperbolic disk by an arithmetic lattice.  Since the details are somewhat involved, we give here only a  rough statement; see Theorem~\ref{theorem:wiman:spinor3} for an exact statement. 

\begin{theorem}[{\bf Arithmeticity of the Wiman curve}]
The Wiman curve is biholomorphic to the quotient of the hyperbolic disk by the principal level $5$ subgroup of a certain ``unit spinor norm'' group of M\"{o}bius transformations.
\end{theorem}

In \S\ref{section:modular} we give a modular interpretation to the Wiman curve $C_0$ itself as a certain moduli space of Hodge structures.  We use this as well as the arithmetic description of $C_0$ to prove the following (see Theorem~\ref{theorem:wiman:modular4a} for a precise statement).

\begin{theorem}[{\bf Modular description of the Wiman curve}]\label{theorem:wiman:modular4}
The Wiman curve naturally supports a family of
abelian surfaces $A$ for which  $\End(A)$ is an indefinite quaternion algebra, which is endowed with an isomorphism $\Fb_5\otimes \End(A)$ with a fixed  $\Fb_5$-algebra and for which $H_1(A; \Zb)$ is a principal $\End(A)$-module. This gives the Wiman curve the structure of a Shimura curve of indefinite quaternionic type.
\end{theorem}

\bigskip
\noindent
{\bf Acknowledgements. } The present paper grew out of our joint work \cite{DFL} with Igor Dolgachev.  It is a pleasure to thank Igor for sharing with us his knowledge and insights on this topic. We also thank Amie Wilkinson for making Figure~\ref{figure:Dodec1}.

\section{Riemann surfaces with $\Af_5$-action and their orbifolds}\label{sect:hypgeom}

\subsection{Reconstructing a surface with symmetry from its orbifold}
This section is closely related to the material in \S 3.2 of \cite{DFL}, but we take here  an orbifold point of view.  This is in a sense a bottom up approach. We will use the language of branched covers and $2$-dimensional orbifolds; see, e.g.,\ \S 2 of 
 \cite{Sc}.  Denote by $(g;m_1,\ldots ,m_r)$ a complex-analytic  (or hyperbolic when specified) orbifold of genus $g$ with $r$ orbifold points of orders 
 $m_1, \dots, m_r$.  We use the classical terminology and call this the {\em genus} of the orbifold. 
 
Our point of departure is  Proposition 3.2 of \cite{DFL}, where we observe that this  proposition is essentially topological in nature. 
It says among other things that  if $C$ is compact Riemann surface of genus $6$ endowed with a faithful action of the alternating group $\Af_5$ (resp.\  the symmetric group $\Sf_5$), 
then it gives rise to an orbifold of  type $(0; 3,2,2,2)$ (resp.\  $(0; 6,4,2)$). We also noted that in the case of an $\Sf_5$-action, the passage from the 
$\Af_5$-orbit space to the  $\Sf_5$-orbit space defines a degree 2 cover from an  orbifold  of type $(0; 3,2,2,2)$ to one
of type $(0; 6,4,2)$ which ramifies over the orbifold points of order $6$ and $4$. 

We investigate here to what extent the orbifold determines the Riemann surface with its group action.  Our findings in this section are summed up in Theorem~\ref{theorem:moduli2} in the introduction.  We will deduce this theorem from two results.  The first reproves in an elementary fashion the uniqueness assertion of Corollary 3.6 of \cite{DFL}.

\begin{theorem}[{\bf Genus $6$ surfaces with $\Sf_5$ action}]\label{theorem:moduli0} 
There is, up to covering transformations, exactly one compact connected Riemann surface of genus $6$ endowed with a faithful $\Sf_5$-action.  
\end{theorem}

The surface whose uniqueness and existence it asserts is called the \emph{Wiman curve}, and we shall denote it by $C_0$.   

In order to state the other two results, let  $(P; q_1,q_2,q_3, q_4)$ be an orbifold of type $(0; 3,2,2,2)$. We first address the question `How many Riemann surfaces of genus $6$ with faithful $\Af_5$-action give rise to this orbifold?'. In order to state our answer, we fix  an embedded segment $\gamma$ in $P\ssm \{q_1, q_2\}$ connecting the order 2 orbifold points $q_3$ and $q_4$.  The embedded segments in $P\ssm \gamma$ connecting $q_1$ with $q_2$ then belong to a single isotopy class; 
let $\alpha$ be such a segment. Clearly, $\alpha$ is a deformation retract of $P\ssm \gamma$.

Let $f:C\to P$  be any connected \emph{smooth} $\Af_5$-covering; by smooth we mean here that $f$ `resolves' the orbifold points in the sense that the order of the orbifold point is the order of ramification of $f$ over this point.  Then $f^{-1}\gamma$ is a disjoint union of embedded circles and
a deformation retraction  of $P\ssm \gamma$  onto $\alpha$ lifts to a deformation retraction  of $C\ssm f^{-1}\gamma$  onto $f^{-1}\alpha$. 
So each connected component of $f^{-1}(P\ssm \gamma)$ has a connected component of $f^{-1}\alpha$ as deformation retract, and 
since  $f^{-1}\alpha$ is a  trivalent graph (with vertex set the preimage of $q_1$), this component  will have first Betti number $\ge 2$, and therefore negative Euler characteristic. 

Denote by $t_\gamma$  the simple braid generator in the mapping class group of the orbifold $P$ supported by  a regular neighborhood of $\gamma$ in $P\ssm \{q_1, q_2\}$. So a homeomorphism representing $t_\gamma$ interchanges $q_3$ and $q_4$.  Note that $t_\gamma$ takes an $f:C\to P$ as above to an 
$\Af_5$-covering  $f':C'\to P$ of the same \emph{type}, by which we mean that this comes with 
orientation-preserving  $\Af_5$-equivariant diffeomorphism $(C, f^{-1}\gamma)\cong  (C', f'{}^{-1}\gamma)$.

The following is the second main result of this section. It answers the above question.

\begin{theorem}[{\bf Genus $6$ surfaces with $\Af_5$ action}]
\label{theorem:moduli1}
Up to $P$-isomorphism, there are 18 $\Af_5$-coverings of $P$ that resolve the orbifold singularities; these coverings are necessarily of genus $6$. They come in three types and make up five orbits under the group $\la t_\gamma\ra$ generated by $t_\gamma$. More precisely, they are:
\medskip

\begin{description}
\item[Petersen configuration] The intersection graph of $(C,f^{-1}\gamma)$ is the Petersen graph; this is a single $\la t_\gamma\ra $-orbit which has two elements.
The $\Af_5$-stabilizer of a connected component of $C\ssm f^{-1}\gamma$ is  conjugate to $\Sf_3^{\ev}$ (or equivalently, is the $\Af_5$-centralizer of a transposition of $\Sf_5$).
\medskip

\item[$K_5$ configuration] The intersection graph of $(C,f^{-1}\gamma)$ is the complete graph on $5$ vertices; it has two $\la t_\gamma\ra $-orbits, each having three elements. The $\Af_5$-stabilizer of a connected component of $C\ssm f^{-1}\gamma$ is  conjugate to the subgroup $\Af_4\subset \Af_5$.
\medskip

\item[Dodecahedral configuration] The intersection graph of $(C,f^{-1}\gamma)$ has a single vertex with five loops; it has two $\la t_\gamma\ra$-orbits, each having five elements (\footnote{The name \emph{dodecahedral configuration} will become clear later.}).
\end{description}
\medskip

An outer automorphism of  $\Af_5$ leaves the orbit type invariant; it exchanges the two elements of the Petersen case and exchanges the two orbits in the last two cases.
\end{theorem}

\begin{remark}   
The topological interpretation of the $t_\gamma$-action is as follows. 
If  $U_\gamma$ is a thin regular neighborhood $U_\gamma$ of $\gamma$, then  each connected component of $f^{-1}U_\gamma$ has the same degree $d$ over $U_\gamma$ and for the three cases listed $d$ equals
$2$, $3$ and $5$ respectively. The mapping class $t_\gamma^d$ takes  $C$ to itself, but induces a Dehn twist along each connected component of $f^{-1}\gamma$. In particular, its restriction to $C\ssm f^{-1}\gamma$ is isotopic to the identity.
\end{remark}

We prove at the same time a universal property. This will yield   a different (and perhaps more elementary) proof of Theorem 3.4 and Corollary 3.5 of \cite{DFL}.

For $P$ as above, the double cover of $P$ branched at the four orbifold points is a genus one Riemann surface, which we make an elliptic curve by taking the order 3 point 
as origin. Then the three other  points of ramification are the points of order $2$ of this elliptic curve. Thus the isomorphism type of $P$ determines and is 
determined by an element of $\PSL_2(\Z)\bs \Hf$, the moduli space of elliptic curves.  

Before proceeding to the proofs of the theorems above, we make a few remarks.

\begin{remark}\label{rem:noncongruence}
Recall that a subgroup of $\PSL_2(\Z)$ is called a {\em congruence subgroup} if  for some positive integer $m$ it is the preimage of a subgroup of $ \PSL_2(\Z/m)$ under the mod $m$ reduction $\PSL_2(\Z)\to \PSL_2(\Z/m)$. It was shown by H.~Larcher (Theorem C in \cite{larcher}) that
the set of cusp widths of such a group is closed under taking \emph{lcm} and \emph{gcd}. As this is evidently not the case 
for the group $\Gamma$ in Theorem \ref{theorem:moduli2}, we conclude that $\Gamma$  is \emph{not} a congruence subgroup.
 \end{remark}

\begin{remark}[{\bf Universal family versus Deligne-Mumford stack}]\label{rem:fivepunctured} 
We established in Theorem 3.4 of \cite{DFL} that $\Bc^\circ$ parametrizes the smooth, projective, genus $6$ curves endowed with a faithful $\Af_5$-action. 
It is given by the part of the base of the Wiman-Edge pencil over which we have smooth fibers; there are five singular fibers, each of which is 
Deligne-Mumford stable, and which have dual intersection graphs as described here. In other words,  our notation  is compatible with the notation 
employed in \cite{DFL}. In particular,  $\Bc$ can be identified with the base of the Wiman-Edge pencil. 

A comment regarding Theorem 3.4 of \cite{DFL}---if not a correction---is in order. 
The $\Sf_5$-action on the Wiman curve implies that as a $\Af_5$-curve it admits an automorphism that is not inner. So if we want to attribute to the 
family over $\Bc^\circ$  a universal property, as does Theorem 3.4, then we must work in the 
setting of Deligne-Mumford stacks. To be precise, while  it is true that every family of smooth  geometrically connected $\Af_5$-curves of genus $6$  fits in a cartesian square
with the Wiman-Edge pencil on the right, there are a few (rare) cases for which there exist more than one choice for the  top arrow.
A similar phenomenon occurs at  the Petersen curve. 
\end{remark}

\begin{remark}\label{rem:fundgroup}
Since $\Gamma$ is torsion free and contained in $\PSL_2(\Z)$, it is in fact free, and it can be identified with the fundamental group of $\Bc^\circ$, with the customary ambiguity: we need to choose a base point. Here a natural choice is the point $c_o\in \Bc$ defining the Wiman curve--and an isomorphism is then given up to inner automorphism. The base $\Bc^\circ$ is a $5$-punctured sphere and so $\Gamma$ is free on 4 generators.
\end{remark}

\begin{figure}
\centerline{\qquad\qquad\includegraphics[scale=0.275]{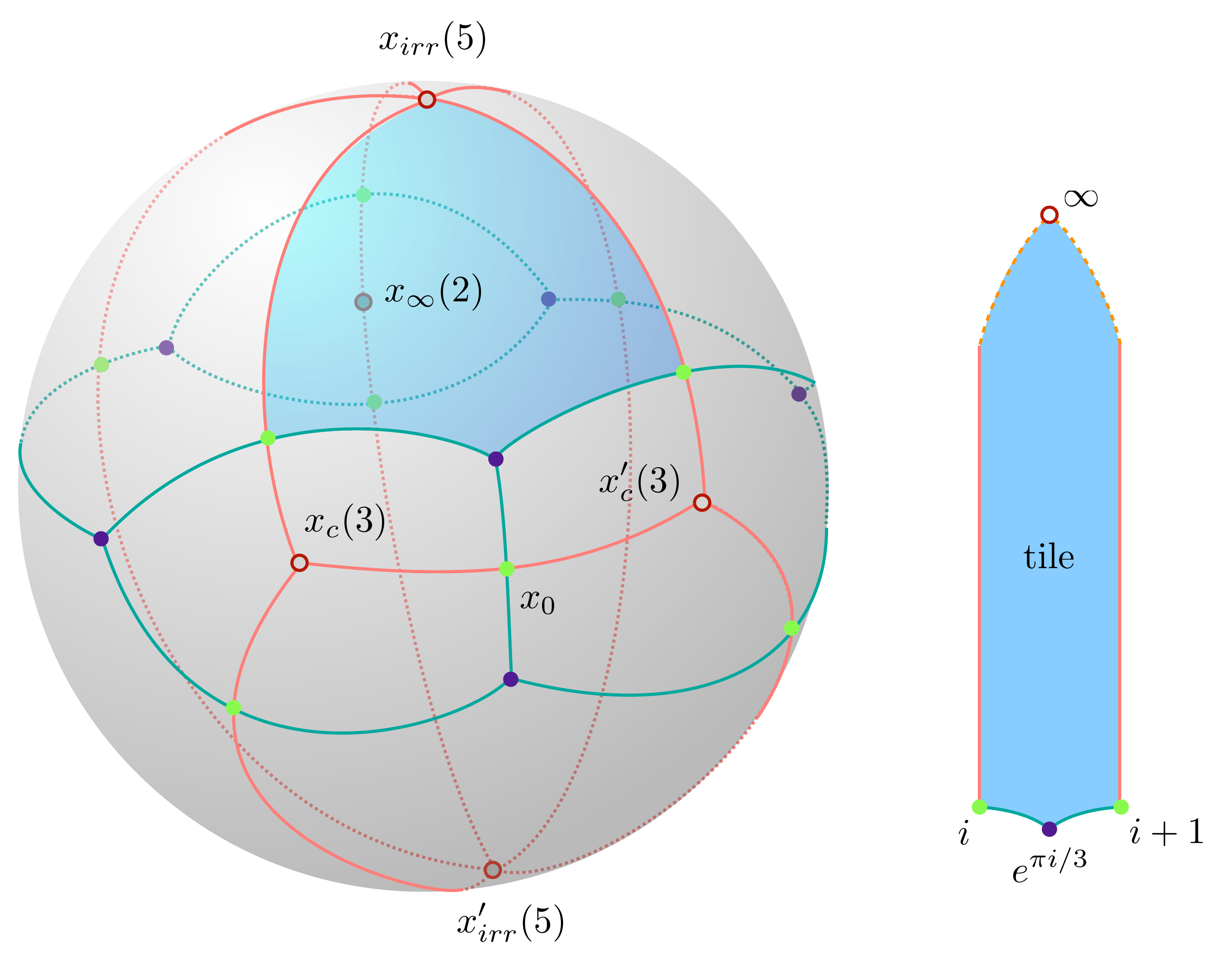}}
\caption{\footnotesize 
The base $\Bc$ of the Wiman-Edge pencil $\Cs$. It is a $2$-sphere with $5$ marked points (colored red in the figure) corresponding to the $5$ singular members of $\Cs$.  As explained in Theorem~\ref{theorem:moduli2}, the complement $\Bc^\circ$ of these $5$ points in ${\Bc}$, which represents the moduli space of smooth genus $6$ curves with an $\Af_5$ action, has a finite area hyperbolic metric with $5$ cusps, one for each deleted point. The cusp $x_\infty$ has width 2, the cusps $x_c$ and $x'_c$ have width $3$ and the 
cusps $x_\ir$ and $x'_\ir$ have width $5$. The point $x_o$ defines the Wiman curve and $\iota$ is the reflection in the axis through $x_o$ and $x_\infty$. }
\label{fig:baseWEpencil}
\end{figure}
In order to prove the theorems stated above, we need some preparation. We do part of this in greater generality than is needed here, as we believe that this makes the discussion more transparent, while it involves  no additional effort.

\subsection{Coverings of orbifolds of genus zero}\label{subsect:G}
We denote by $\pi_{0,n}$  the group with generators $a_1, \dots, a_n$ subject to the relation $a_1\dots a_n=1$. 
This  is of course the fundamental group of an $n$-punctured genus zero surface, where $a_i$ represents a simple loop around the $i$th puncture.  We begin by introducing notation for data that are 
useful for encoding a finite group action on a Riemann surface.   

\begin{definition}[The set {\boldmath$G(p)$}]
\label{definition:G}
Given a group $G$ and a  descending sequence of positive integers 
$\pbf=(p_1\ge  \dots \ge p_n)$  of length $n\ge 3$ with $p_n>1$, we denote by $\tilde G(\pbf)$ the  set of surjective group homomorphisms $\gbf: \pi_{0,n}\to G$ such that $\mathbf{g}(a_i)$ has order $p_i$ ($i=1, \dots , n$).
We can of course think of $\gbf$ as an $n$-tuple $(g_1, \dots, g_n)\in G^n$ with $g_1g_2\dots g_n=1$ and whose members generate $G$.
Note that $\Aut(G)$  acts on $\tilde G(\pbf)$ by postcomposition  on the left. In particular $G$ acts via $\In(G)$ on $\tilde G(\pbf)$ by simultaneous conjugation.  Define $G(\pbf)$ to be the quotient
\[
G(\pbf):=G\bs\tilde G(\pbf)
\]
so that the $\Aut(G)$-action  on $\tilde G(\pbf)$ induces an $\Out(G)$-action on $G(\pbf)$. 
\end{definition}

Denote by $\widetilde\Mod_{0, \pbf}$  the group of automorphisms of $\pi_{0,n}$ that permute the conjugacy classes of the $a_i$ in such a manner that if $a_i$  is mapped to the conjugacy  class of $a_j$, then $p_i=p_j$.  It acts in an obvious manner (by precomposition, so on the right) on  $\tilde G(\pbf)$, and this action evidently commutes with the $\Aut(G)$ action. Note that  $\widetilde\Mod_{0, \pbf}$ contains the group of inner automorphisms of $\pi_{0,n}$  as a normal subgroup.
We write $\Mod_{0, \pbf}$ for the quotient group. So $\Mod_{0, \pbf}$  is a group of outer automorphisms of $\pi_{0,n}$; we  may regard it as the mapping class group of an orbifold of type $(0; p_1, \dots , p_n)$.

The outer automorphisms of $\pi_{0,n}$ that preserve the conjugacy class of \emph{each} $a_i$ form a subgroup of $\Mod_{0, \pbf}$ that can be 
identified with the mapping class group  $\Mod_{0,n}$ of an $n$-punctured sphere that preserves the punctures.  It is a normal subgroup of $\Mod_{0, \pbf}$  with factor group the subgroup $\Sf_{\pbf}\subset \Sf_n$ of $\sigma\in \Sf_n$ which fix the map $i\mapsto p_i$, so that we have a short exact sequence
\[
1\to \Mod_{0,n}\to \Mod_{0, \pbf}\to \Sf_{\pbf}\to 1.
\]
These  groups also have an interpretation as (orbifold) fundamental group: $\Mod_{0,n}$ is the fundamental group of the fine moduli space $\calM_{0,n}$ of $n$-pointed genus zero curves (the space of injective maps $\{1, \dots, n\}\to \Pb^2$ modulo projective equivalence) and $\Mod_{0, \pbf}$ is the orbifold fundamental group of the space $\calM_{0, \pbf}$ parametrizing orbifolds of type $(0; p_1, \dots, p_n)$. The  latter is in general not a fine moduli space, but underlies a Deligne-Mumford stack.

For any $\gbf\in \tilde G(\pbf)$,  precomposition of $\gbf$ with the inner automorphism defined by $a\in \pi_{0,n}$ is the same 
as postcomposition with the inner automorphism defined by $\gbf(a)\in \pi_{0,n}$,  and so the right action of $\widetilde\Mod_{0, \pbf}$ on
$\tilde G(\pbf)$ decends to a right action of $\Mod_{0, \pbf}$ on $G(\pbf)$ that commutes with the action of  $\Out(G)$.

We use this setup to  encode the data of a $G$-cover of an orbifold of type $(0; p_1, \dots , p_n)$. We shall see that the set $G(\pbf)/\Mod (\mathbf {p})$ classifies the closed connected surfaces with a faithful $G$-action for which the associated orbifold is of type $(0; p_1, \dots , p_n)$. Points of $G(\pbf)/\Mod_{0,n}$ 
enumerate in addition the irregular orbits (the $i$th orbit has size $|G|/p_i$).

\subsubsection*{Constructing the moduli space} We now explain  how one constructs the entire moduli space of Riemann surfaces  endowed  with a faithful action of a finite group $G$ whose associated orbifold is of type $(0; p_1, \dots, p_n)$. We shall here assume that $n\ge 3$ and that each $p_i$ is $\ge 2$. Let $P$ be a copy of the Riemann sphere. Given an injection  
$\qbf: \{1, \dots, n\}\to P$, $i\mapsto q_i$,  put $U_\qbf:=P\ssm\{q_1,\dots, q_n\}$.  Choose a  basepoint $q_0\in U_\qbf$ and a set of standard generators of $\pi_1(U_\qbf, q_0)$ that identifies $\pi_1(U_\qbf, q_0)$ with $\pi_{0,n}$. So   $a_i$ is represented by a simple (positively oriented) loop around 
$q_i$  ($i=1, \dots, n$) such that  $a_1\dots a_n=1$ is the only relation (we traverse composite loops in the given order, so from left to right). 

Given $\mathbf{g}\in \tilde G(\pbf)$, then a homomorphism $\pi_1(U_\qbf, q_0)\to G$  is defined by $a_i\mapsto g_i$. This homomorphism is surjective and hence defines a connected 
$G$-covering $U_\qbf(\mathbf{g})\to U_{\qbf}$. By the theory of  coverings, $\gbf'\in \tilde G(\pbf)$ is in the same  $G$-orbit as $\gbf$ precisely if there exists a $G$-isomorphism $U_\qbf(\mathbf{g'})\cong U_\qbf(\mathbf{g})$ over $U_\qbf$.  This isomorphism need not be unique, because  a $G$-covering 
$U_\qbf(\mathbf{g})\to U_\qbf$ may have automorphisms over $U_\qbf$. But such an automorphism must be a deck transformation which commutes with all other deck transformations, in other words, must belong to the center of $G$. Therefore, \emph{we assume from now on that $G$ has trivial center.}  Note that this holds both for $\Af_5$ and for $\Sf_5$.

We apply the Riemann  extension theorem to extend  $U_\qbf(\mathbf{g})\to U_{\qbf}$ uniquely  
to a ramified $G$-covering  $P_\qbf(\mathbf{g})\to P$, so that $P_\qbf(\mathbf{g})$ is a nonsingular projective curve and with ramification order $p_i$ over  $q_i$. 
Because of the above assumption, a $G$-isomorphism $P_\qbf(\mathbf{g'})\cong P_\qbf(\mathbf{g})$ exists over $P$ if and only if $\gbf'$ and $\gbf$ lie in the same $G$-orbit. So if we regard
 $(P,\qbf)$ as an orbifold of type $(0; p_1, \dots, p_n)$, then  the smooth connected $G$-covers of $(P,\qbf)$ (so  those which resolve the orbifold singularities) form a $G(\pbf)$-torsor.

If we now allow $\qbf$ to vary, then the $P$-isomorphism types  of these covers define a covering of the configuration space $P^{(n)}$ of injective maps  
$\qbf: \{1, \dots, n\}\hookrightarrow P$ with  fiber  isomorphic to $G(\pbf)$. The action of $\Aut (P)$ on $P$  lifts to this covering, despite the fact that  
$\Aut (P)\cong\PSL_2(\C)$ is not simply-connected.  The reason is that traversing a  nontrivial loop in $\Aut (P)$ based at the unit acts on
$\pi_1(U_\qbf, q_0)$ as an inner automorphism and so will act trivially on the cover. When we divide out by this $\Aut (P)$-action (which is easily accomplished by fixing the first three components of $P^{(n)}$) we  obtain a diagram
\begin{equation}\label{eqn:universal}
\Cs_{\widehat{B} G (\pbf)}\to\widehat{B}G(\pbf)\to\calM_{0,n},
\end{equation}
in which  the first morphism has the interpretation of a fine moduli space of systems  
$(C,\phi;\Oc_1,\dots , \Oc_n )$, where $C$ is a  connected compact Riemann surface,  $\phi: G\hookrightarrow \Aut(C)$ an injective homomorphism such that $G\bs C$ has genus zero and $\Oc_1,\dots , \Oc_n$ is a faithful 
enumeration of the irregular orbits such that  $\Oc_i$  has order $|G|/p_i$. Here  we declare $(C,\phi;\Oc_1,\dots , \Oc_n )$ and 
$(C',\phi';\Oc'_1,\dots , \Oc'_n  )$ to be equivalent if there  exist a $G$-isomorphism of  $C$ onto $C'$ over $P$ that respects the indexing of the irregular orbits. 
Since $G$ is centerless,  this isomorphism is unique, so that the fibers of $\Cs_{\widehat{B} G (\pbf)}\to\widehat{B}G(\pbf)$ are smooth curves endowed
with a $G$-action. The second morphism is given by formation of the $G$-orbifold, where we use the numbering of the irregular orbits.

The  finite group $\Sf_{\pbf}$ acts on $P^{(n)}$ by permuting the factors. This  commutes with the $\Aut(P)$-action, but the product 
action need no longer be free, for some nontrivial element in $\Aut(P)$  could permute nontrivially the ramification points  $(q_1, \dots , q_n)$, while 
preserving  their weights. In other words, the residual action of the finite group $\Sf_{\pbf}$ on the diagram (\ref{eqn:universal}) may have fixed points in $ \calM_{0,n}$ and  
can act nontrivially on a fiber over such a point. So the  quotient by this action is a priori a Deligne-Mumford stack
\begin{equation}\label{eqn:universal2}
\Cs_{\Bb G (\pbf)}\to \Bb G(\pbf)\rightarrow \Mb_0(\pbf).
\end{equation}
Its modular interpretation is that of (\ref{eqn:universal}), except that  the irregular orbits are no longer numbered. Observe that a fiber of  $\Bb G(\pbf)\rightarrow \Mb_0(\pbf)$  over a non-orbifold point is  as a $\Out(G)\times \Mod(\pbf)$-set  identified with $G(\pbf)$, where the action $\Mod(\pbf)$  is as a group of covering transformations of $\Bb G(\pbf)\rightarrow \Mb_0(\pbf)$, 
so that the connected components of $\Bb G(\pbf)$ are indexed by the set $G(\pbf)/\Mod(\pbf)$. 

\begin{remark}\label{rmk:DM}
We have to resort to Deligne-Mumford stacks, because there might exist a $G$-curve admitting an automorphism which
nontrivially permutes  its irregular orbits. We shall see that this happens for the Wiman curve.
\end{remark}

\subsection{Proofs of the three theorems} 
In view of the previous discussion, Theorem~\ref{theorem:moduli0}  is an immediate consequence to the following lemma.
  
\begin{lemma}\label{lemma:s5}
The set $\Sf_5(6,4,2)$ (in the notation of \S\ref{subsect:G}) is a principal  $\Sf_5$-orbit. An orbit representative is
$((123)(45), (1245),(14)(23))$. 
 \end{lemma}
 \begin{proof}
Since the elements of $\Sf_5$ of order $6$ make up a single conjugacy class, we can assume that $g_1=(123)(45)$.  Suppose that $g_2,g_3\in\Sf_5$ is such that $(g_1,g_2,g_3)$ satisfy the definition of $\Sf_5(6,4,2)$.  Note that $g_2$ must be a $4$-cycle.  We must show that any two choices of $g_2$ are conjugate via an element of the centralizer $Z(g_1)$ of $g_1$ in $\Sf_5$.  Note that 
$Z(g_1)$ is generated by $g_1$.  The $Z(g_1)$-conjugacy classes of $4$-cycles are represented by $(1234)$, $(1324)$, $(1245)$ and $(2145)$. If $g_2$ is one of these elements, then we find that $g_1g_2$ is of order $5$, $3$, $2$ and $3$ respectively. So $(1245)$ is the unique representative with the required property.  We then find  
that $g_3=(14)(23)$. 
\end{proof}

We now turn to the proofs of Theorems~\ref{theorem:moduli1} and \ref{theorem:moduli2}. So here $n=4$ and $(p_1,p_2, p_3, p_4)=(3,2,2,2)$.   We begin with three combinatorial lemmas.

\begin{lemma}\label{lemma:a5}
Let $\Af_5(r)$ denote the set of elements of $\Af_5$ of order $r$. Consider the action of 
$\Af_5$ on $\Af_5(3)\times \Af_5(2)$ by simultaneous conjugation.  Then this action is free, and every orbit is represented by precisely one of the following pairs (we name each case after the conjugacy class of the subgroup of $\Af_5$ generated by the pair):

\begin{description}
\item[$\Sf_3^{\ev}$] $((123), (23)(45))$, so that any  pair $(g_1, g_2)$ in this orbit generates a subgroup  conjugate to $\Sf_3^{\ev}$ and  $g_1g_2$ has order $2$,
\smallskip

\item[$\Af_4$] either $((123), (12)(34))$ or  $((123), (12)(35))$, so that  any  pair $(g_1, g_2)$ in one of these  orbits generates a subgroup  conjugate to
$\Af_4$ and $g_1g_2$ has order $3$,
\smallskip

\item[$\Af_5$]  either $((123), (14)(25))$ or $((123), (15)(24))$, so that  any  pair $(g_1, g_2)$ in one of these  orbits generates a subgroup  conjugate to
$\Af_5$ and $g_1g_2$ has order $5$.
\end{description}
 \end{lemma}
 \begin{proof}
It is helpful to picture a $k$-cycle  of $\Sf_r$ as an oriented $k$-polygon with vertex set a $k$-element subset of 
$\{1, \dots, r\}$ (for $k=2$ this amounts to an unoriented edge).    So an element $(g_1,g_2)\in \Af_5(3)\times \Af_5(2)$  is represented by  oriented  triangle and an unordered pair of  disjoint edges.  We then see that only three  types are possible:
\medskip

\begin{description}
\item[$\Sf_3^{\ev}$]  the  pair of  disjoint edges meets the triangle in an edge only (the other edge is then the unique edge disjoint with  the triangle),
\smallskip

\item[$\Af_4$] the  pair of  disjoint edges  meets the triangle in the union of an  edge  and  the opposite vertex,
\smallskip

\item[$\Af_5$]  the  pair of  disjoint edges meets the triangle in two distinct vertices.
\end{description}
\medskip

This graph-like classification (where we ignore the labeling of the vertices) is a geometric way of describing  the $\Sf_5$-orbits in $ \Af_5(3)\times \Af_5(2)$. So there are three of these.
It is straightforward to check that in the first case the $\Sf_5$-orbit is also a   $\Af_5$-orbit, and in the two other cases splits into two such orbits.
The other assertions are also straightforward.
\end{proof}

\begin{lemma}\label{lemma:a5'}
Fix $h\in \Af_5(r)$ and consider the set of pairs $(h_1,h_2)\in \Af_5(2)\times \Af_5(2)$ such that $h=h_1h_2$.  For $r=2,3,5$, this set consists of exactly $r$ items: for $r=2$, $h_1$ and $h_2$ must commute and the two pairs only differ by their order (they generate a Kleinian Vierergruppe), whereas for $r=3$ and $r=5$ this is a free $\la h\ra$-orbit (acting by simultaneous conjugation). 
\end{lemma}

The proof of this lemma  is left as an  exercise.

\begin{proposition}\label{prop:a5}
The $(g_1, g_2, g_3, g_4)\in\Af_5(3,2,2,2)$  come in  three types, according to the order $r$ of $g_1g_2$: two elements with $r=2$, six elements with $r=3$ 
and ten elements with $r=5$. The stabilizer of $a_1a_2$ in  $\Mod (3,2,2,2)$ preserves  each type, 
acts  transitively on the set of elements with $r=2$, and has two orbits for $r=3, 5$.  

The nontrivial element of  $\Out(\Af_5)\cong \Sf_5/\Af_5$ acts on $\Af_5(3,2,2,2)$  by exchanging the last two items for $r=2$ and 
exchanging the  two $\Mod (3,2,2,2)$-orbits for $r=3, 5$. 

The action of $\Mod (3,2,2,2)$ on $\Af_5(3,2,2,2)$ is transitive. 
\end{proposition}

\begin{proof}
For the first assertion we must count the $\Af_5$-orbits in $\tilde\Af_5(3,2,2,2)$.  Lemma \ref{lemma:a5} allows us  represent each orbit 
by a  $\gbf=(g_1,g_2,g_3, g_4)$ with $(g_1, g_2)$ as in that lemma. This leads us to the three types, according to the order  $r$ of $g_1g_2$: 
for $r=2$ we have one case and for $r=3,5$ we have two cases to consider. Lemma \ref{lemma:a5} also asserts that simultaneous $\Af_5$-conjugation
acts freely on $\Af_5(3)\times \Af_5(2)$ and so for a given $(g_1, g_2)\in \Af_5(3)\times \Af_5(2)$,  it remains to count the number of  $(g_3,g_4)\in\Af_5(2)\times \Af_5(2)$ for which  $g_3g_4=(g_1g_2)^{-1}$. This is what
Lemma \ref{lemma:a5'} does for us,  and thus the first assertion  follows.  Note that a transposition in $\Sf_5$ (which represents the nontrivial element of 
$\Out(\Af_5$) 
exchanges the two pairs $(g_3, g_4)$ that we get for $r=2$ and exchanges the two orbits that we get for $r=3, 5$. 

The stabilizer of $a_1a_2$ in $\Mod (3,2,2,2)$ will of course preserve the type. One such element is given by 
\[
(a_1,a_2, a_3, a_4)\mapsto (a_1,a_2, (a_1a_2)a_3(a_1a_2)^{-1}, (a_1a_2)a_4(a_1a_2)^{-1}).
\] 
This has on $\Af_5(3,2,2,2)$ of course a similar effect with 
$a_i$ replaced by $g_i$. Lemma \ref{lemma:a5'} shows that for $r=3$ and $r=5$ we get all the elements in  $\tilde\Af_5(3,2,2,2)$  with $(g_1, g_2)$ prescribed  and of order $3$ or $5$. When $r=2$ we use the element of $\Mod (3,2,2,2)$ given by 
$(a_1,a_2, a_3, a_4)\mapsto (a_1,a_2, a_3a_4 a_3^{-1}, a_3)$, which has the effect on $(g_1,g_2,g_3, g_4)$ of exchanging $g_3$ and $g_4$.
This proves the second assertion.

The last assertion will follow if we prove that $\Mod (3,2,2,2)$ acts transitively on the types. We start with  $\gbf=((123), (14)(25), (12)(34), (15)(24))$, for which  the product of the first two items has order $5$. The element of $\Mod (3,2,2,2)$ defined by  
\begin{align*}
(a_1,a_2, a_3, a_4)&\mapsto (a_1, a_3, a_3^{-1}a_2 a_3,a_4) \text{ and its inverse }\\ (a_1,a_2, a_3, a_4)&\mapsto (a_1, a_2a_3a_2^{-1}, a_2,a_4)
\end{align*}
takes  $\gbf$ to  $((123), (12)(34), \dots)$ resp.\ $((123), (13)(45), \dots)$ for which the product of the first two items has order $3$ resp.\ $2$.
\end{proof}

\begin{proof}[Proof of Theorem~\ref{theorem:moduli1}]
The number and the labeling  of  $18$ cases follow from Proposition \ref{prop:a5}. It remains to show that  the intersection graphs are as asserted.
We do this with the help of Lemma \ref{lemma:a5}: it tells us that the vertices are in bijection the $\Af_5$-left cosets of $\la g_1, g_2\ra\subset \Af_5$
and the edges are in bijection with the $\Af_5$-left cosets of the cyclic subgroup  $\la g_1g_2\ra$, the incidence relation being given by inclusion.
This is then straightforward.
\end{proof}

\begin{proof}[Proof of Theorem~\ref{theorem:moduli2}]
Let us first observe that local charts of the moduli space $\Bc^\circ$ of compact, connected Riemann surfaces of genus $6$ with faithful $\Af_5$-action are obtained by a total order on the four orbifold points and taking their cross ratio. This implies
that the map $\Bc^\circ\to \PSL_2(\Z)\bs \Hf$ resolves the orbifold singularities of $\PSL_2(\Z)\bs \Hf$. In other words, this identifies 
$\Bc^\circ$ with a quotient $\Gamma\bs \Hf$, where $\Gamma\subset \PSL_2(\Z)$ is a torsion-free subgroup. The index  of $\Gamma$ in $\PSL_2(\Z)$
will of course be the degree of  $\Bc^\circ\to \PSL_2(\Z)\bs \Hf$, i.e., $18$. It follows from Proposition \ref{prop:a5} that $\Gamma$ has five cusps with width $2$, $3$, $3$, $5$ and $5$. The Riemann-Hurwitz formula  shows that if we fill in each of the five cusps, the resulting Riemann surface $\PSL_2(\Z)\bs \widehat{\Hf}$ has genus zero. 

The degenerations in question are obtained by shrinking $\gamma$  (for this is how we tend to the  cusp of $\PSL_2(\Z)\bs \Hf$). The extension of 
$\Cs_{\Bc^\circ}\to{\Bc^\circ} $ to   $\Cs_{\Bc}\to \Bc$ is formally taken care of by the general theory of Hurwitz schemes \cite{BR}, 
but as the present case is a relatively simple instance of this, we briefly indicate how this is done. By shrinking $\gamma$ we make of course also the connected components of $f^{-1}\gamma$ shrink and each such a `vanishing component' creates a node. To be precise, a  
local model at such a point is the double cover $w^2=z^2-t$ (with ramification points the two roots of $z^2-t$) and when $t$ moves in the unit disk, then over $t=0$ the cover acquires an ordinary node.  In global terms, an $\Af_5$-cover of $(P; q_1,q_2, q_3, q_4)$ becomes a stable curve when $q_3$ and $q_4$ coalesce and its dual intersection graph is then as described above.  That these graphs are as asserted then follows from Theorem~\ref{theorem:moduli1}. 

Applying an outer automorphism of $\Af_5$ must define an involution $\iota$ in $\Gamma\subset \PSL_2(\Z)$. It will have the point of $\Bc^\circ$ defining the Wiman 
curve as fixed point  for which the orbifold is representable $(\Pb^1; \infty, -1,0,1)$. This involution is therefore  representable by a M\"obius transformation 
in the $\PSL_2(\Z)$-conjugacy class of  $\left( \begin{smallmatrix}0 &1 \\ -1 & 0\end{smallmatrix}\right)$.  It must 
normalize  $\Gamma$ and so it generates with $\Gamma$ a subgroup   $\tilde\Gamma\subset\PSL_2(\Z)$ which contains $\Gamma$ as a subgroup of 
index $2$. The involution $\iota$ fixes the cusp of width $2$, interchanges the two cusps of width $3$, and interchanges the two cusps of width $5$.
\end{proof}

\begin{remark}   
The theorems in this section do not include the assertion that Figure \ref{fig:baseWEpencil} is the correct representation of $\Gamma\bs\widehat\Hf$. We know however that the latter  has the structure of a polygonal complex with the solid  $n$-gons  in bijective correspondence with the cusps of width $n$
in such a manner that in each vertex exactly three of these meet.  In our case we have two pentagons, two triangles and  one `bigon'  and then the reader will easily find that this figure is essentially the only way these can fit together to yield a closed surface (necessarily of genus zero). 
\end{remark}

\section{Geometric models of the Wiman-Edge pencil}\label{sect:combmodel}

In this section we first give a concrete geometric picture of each type of  degeneration of the Wiman-Edge pencil.
This is followed by three geometric descriptions of the Wiman-Edge pencil: the first one is directly based on the main results of Section \ref{sect:hypgeom} and the other two come from putting a natural piecewise flat resp.\  hyperbolic structure  on the fibers.

\subsection{Geometric models for the three degeneration types}
Recall that we have essentially three types of nodal degenerations for genus 6 curves with $\Af_5$-symmetry, distinguished  by their dual intersection graphs,
or what amounts to the same, their configurations of vanishing cycles: the Petersen curve, the $K_5$-curve and the dodecahedral curve. For each of these cases we give a geometric model and use polar  coordinates (of Fenchel-Nielsen type)  to describe the degeneration.

For the Petersen and the $K_5$-curve this is accomplished  by  the boundary of a 
regular neighborhood of an embedding  the graph in real $3$-space.
Over each midpoint of an edge we have a (vanishing) circle and these circles decompose the surface into ten $3$-holed spheres
resp.\ five $4$-holed spheres, see Figure (\ref{figure:degenerations}).  The $\Af_5$-stabilizer of such a holed sphere $S$ acts transitively on its boundary components and is a copy of $\Sf_3$ resp.\ $\Af_4$. 

Thus a $\Af_5$-invariant hyperbolic structure on the surface such that each vanishing circle becomes a geodesic, gives on $S$ a hyperbolic metric invariant for which $\partial S$ is geodesic and invariant under its stabilizer group.  All boundary components of $S$ thus have the same length $r$, and it is not hard to check that $r$ is a complete invariant of the isometry type of $S$. Conversely, if we are given $r>0$, then a number of copies of a hyperbolic holed sphere with symmetry and geodesic boundary length $r$  can be assembled to produce a closed hyperbolic surface with $\Af_5$-symmetry, although this involves an additional (angular) parameter $\theta$ which prescribes how boundary components of spheres are identified. The degeneration in question is obtained by letting $r$ tend to zero.

For the irreducible degeneration we proceed as follows.
Let  $D$ be a regular dodecahedron and choose an isomorphism of its isometry group  with $\Af_5$. 
Then the  group $\Af_5$ acts transitively on the set of $12$ faces of $D$, each of which is a regular pentagon. 
From each pentagonal face of $D$ we remove an open disk of a small radius (the same for every face)  centered at its barycenter so  that the resulting surface $D'$ is a $\Af_5$-invariant 12-holed sphere. Next we identify opposite boundary components of $D'$ in a $\Af_5$-equivariant manner, for example by means of the antipodal map. Since there are 6 such pairs,  the resulting  surface $\Sigma$ has genus $6$ and comes with a faithful action of $\Af_5$ by orientation preserving homeomorphisms.
Notice that the image of each face of $D'$ makes up a  cylinder in $\Sigma$.  These $12$ cylinders pave $\Sigma$ in particular way: one boundary component of a cylinder meets 5 others (let us call this the \emph{pentagonal component}) and the other  boundary component  (the \emph{smooth  component}) meets it opposite copy.

The surface $\Sigma$  also inherits from $D$ a flat metric which has the 20 vertices of $D$ as its cone points, each with angular defect $\pi/5$. 
The underlying conformal structure extends across such points and makes it a Riemann surface with $\Af_5$-action. The interior of each cylinder will be even a cylinder in the conformal sense, but its closure has only symmetry under a cyclic group of order 5: along  the pentagonal  component  it has now at each vertex of the pentagon an  interior angle of size $2\pi/3$. A fundamental domain  for this symmetry group is conformally equivalent  to a hyperbolic 4-gon with angles
$\pi/2, \pi/2, \pi/3, \pi/3$. Its isometry type is determined by the length $\ell$ of its first edge (between the straight angles). Note that it has an obvious axis of symmetry. Five copies of this $4$-gon make up a cylinder $Z(r)$ of the desired type; the geodesic boundary component has length $r=5\ell$.
We have no freedom in the way these 12 copies of $Z(r)$ are  glued onto one another along their  pentagonal components, so that $r$ is a complete invariant of the 12-holed  sphere. But we have an additional angular parameter $\theta$ when identifying opposite boundary components. The degeneration is defined by letting $r$ tend to zero.

In all three cases, the pair $(r,\theta)$ is a pair of Fenchel-Nielsen coordinates.
 
\begin{figure}
\centerline{\qquad\qquad\includegraphics[scale=0.5]{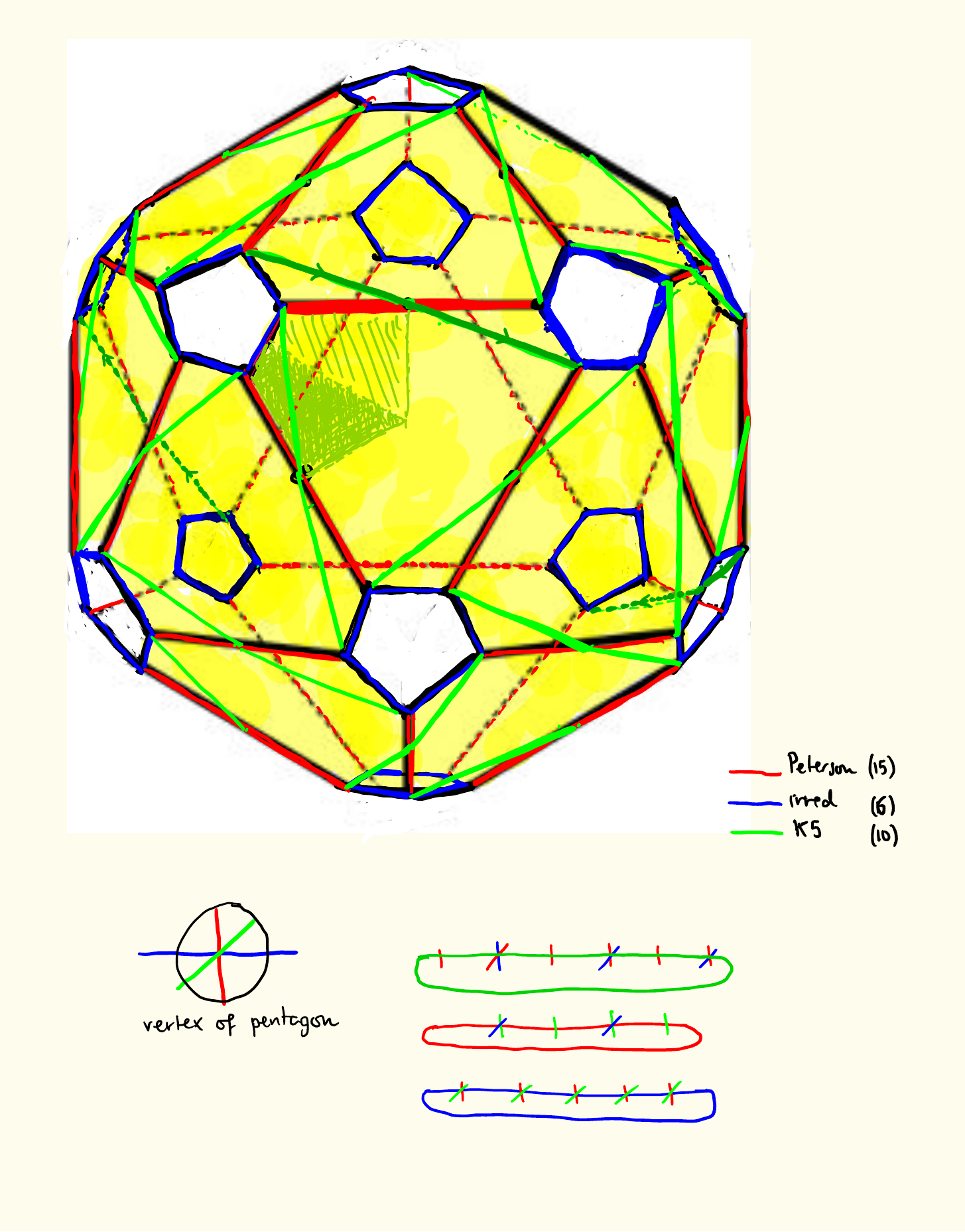}}
\caption{\small The three types of vanishing cycles depicted on an icosahedral surface of genus 6 (opposite pentagons must be  identified).}\label{figure:vc}
\end{figure}

\subsection{Geometric interpretation of the tesselation of $\Bc$}
Let $P$ be a complex orbifold of type $(0;3,2,2,2)$. We observed that if we choose an embedded segment $\gamma$ that connects two orbifold points of order $2$, then for every smooth $\Af_5$-cover $f: C\to P$, $f^{-1}\gamma$ is a closed submanifold of dimension one which decomposes $C$ into genus zero surfaces with boundary. If  $\gamma'$ is another  such embedded segment 
that meets $\gamma$ in just one orbifold point, but has a different tangent direction half-line at that point, then the two one-dimensional submanifolds 
$f^{-1}\gamma$  and $f^{-1}\gamma'$ will meet transversally. We show how to do this in a prescribed and consistent manner so that we end up with a combinatorial model of sorts  of the Wiman-Edge pencil.

Denote by $q_\infty\in P$ the order $3$ orbifold point.  We regard  $P\ssm\{q_\infty\}$ is an affine complex line. The  set $Q$ of three orbifold points of $P$ of order $2$ then make up a (possibly degenerate) triangle in this affine line.

Observe that the group $\PSL_2(\Z)$ has in $\Hf$ the fundamental domain the hyperbolic quadrangle (with one improper point at infinity) $F\subset\Hf$ defined by  
\[
    F=\{z: 0\le \Re (z)\le 1, |z|\ge 1 \ \text{and}\ |z-1|\ge 1\}.
\]
This gives rise to a tesselation of $\Hf$ with two $\PSL_2(\Z)$-orbits of vertices, namely of $\rho:=\exp(\pi \sqrt{-1}/3)$ and of 
$\sqrt{-1}$; and  two $\PSL_2(\Z)$-orbits of edges, namely of the (bounded) arc $E_b$ centered at $0$ connecting $\rho$ and $\sqrt{-1}$, and the unbounded half line $E_u$ on the imaginary axis with imaginary part $> 1$. Observe that  $F$ is bounded by $E_u, E_b$ and the translates 
$E'_b:= \left( \begin{smallmatrix}1 &-1 \\ 1 & 0\end{smallmatrix}\right)E_b$ and 
$E'_u:= \left( \begin{smallmatrix}1 &1 \\ 0 & 1\end{smallmatrix}\right)E_u$.  
We also note that  the  $\PSL_2(\Z)$-orbit of $\overline{E}_b$ is  the familiar trivalent graph that is a $\PSL_2(\Z)$-equivariant deformation retract of 
$\Hf$. So its image in $\Gamma\bs \Hf$  (which we shall denote by $K$) will be  a deformation retract of $\Gamma\bs \Hf$.

When we interpret $\PSL_2(\Z)\bs \Hf$ as the coarse moduli space of elliptic curves, the points of 
$E_b$ and $E_u$ represent elliptic curves that admit a {\em real structure}; that is, elliptic curves that have an antiholomorphic involution.  Such an antiholomorphic involution is here given as an antiholomorphic involution of $P$. It is then determined by its fixed-point set in $P$. This is a real form of $P$ (a real projective line). In our 
case this fixed-point set must pass through $q_\infty$, and so is given in $P\ssm\{q_\infty\}$ as a straight line. To see this, let us enumerate 
the members $\{q_1,q_2, q_3\}$ of  $Q$   in such a manner that  $[q_2, q_3]$ 
is a side of minimal length and $(q_1,q_2, q_3)$ defines the complex orientation of $P\ssm\{q_\infty\}$ when the triangle is nondegenerate. This labeling is almost always unique; the situations for which this is not the case correspond to edges or vertices of $F$, namely: 
$E_u$ parameterizes the situations for which $q_3$ lies on the segment $[q_1, q_2]$;  $E'_u$  for when $q_2$ lies on the segment $[q_1, q_3]$;  
$E_b$ for when $q_3$ is equidistant to $q_1$  and $q_2$;  $E'_b$ for when $q_2$ is equidistant to $q_1$ and $q_3$, and (hence)
$\rho$ corresponds to $q_1, q_2,q_3$  making up an equilateral triangle, $i$ to the case when $q_3=\frac{1}{2}q_1+\frac{1}{2}q_2$,
and $i+1$ the situation for which $q_2=\frac{1}{2}q_1+\frac{1}{2}q_3$. In these cases the line defining the  antiholomorphic symmetry is obvious: it is either a line spanned by two points of $Q$ or a line of points equidistant to two points of $Q$.

\subsection{A model based on affine geometry}\label{subsect:combmodel1}
The three sides of $Q$ have a well-defined length ratio. Denote by $\gamma_P$ the union of the sides of this triangle of minimal length.  So either $\gamma_P$ consists of a single edge, or consists of two edges of equal length making an angle $>\pi/3$, or consists of the three sides of an equilateral triangle. So over $E_b$, $\gamma (P)$ adds to $[q_2, q_3]$  the side $[q_1, q_3]$,  over $E'_b$ the side $[q_1, q_3]$ and hence  over $\rho$ we have  all three sides.  The map that assigns to $P$ the isomorphism type of  $\gamma_P$ defines on the moduli space of such orbifolds a `spine' which, as we will now explain,  is familiar via the identification of this moduli space with $\PSL_2(\Z)\bs\Hf$.

If $f: C\to P$ is an $\Af_5$-cover and $\gamma_P$ consists of one edge, then the preimage $f^{-1}\gamma_P$  is as described by Theorem 
\ref{theorem:moduli1}: this is a codimension one submanifold that decomposes $C$ into genus zero surfaces-with-boundary. But if 
$\gamma_P$ consists of two edges, then a deformation retraction $P\ssm\gamma_P$ onto $q_\infty$ lifts to $C\ssm f^{-1}\gamma_P$, and this shows that every connected component of $C\ssm f^{-1}\gamma_P$ is a topological disk. If $\gamma_P$ consists of three sides, then a similar argument shows that the same is true for a connected components of $C\ssm f^{-1}\gamma_P$.

\subsection{A model based on piecewise flat structures}\label{subsect:JSmodel}
The theory of Jenkins-Strebel differentials provides another combinatorial model of the Wiman-Edge pencil. This amounts to putting on the complement of the size $20$ orbit of every fiber a complete flat structure with a finite number of singularities.  

Let $P$ be an orbifold $P$ of type $(0; 3,2,2,2)$ in the holomorphic category. Denote  the orbifold point of order $3$  by $q_\infty$ and 
regard $P^\circ:=P\ssm\{ q_\infty\}$ as an affine complex line. We denote the 3-element set of order 2 orbifold points by $Q$. For any $a\in P^\circ$  there exists a unique meromorphic quadratic differential  $\eta_{P,a}$ on $P$ characterized by the property that  it has divisor $-(2q_\infty)-Q+a_P$  and with double residue at $q_\infty$ equal to $-1$.
So if we choose an affine coordinate $z$ for $P^\circ$, then
\[
\eta_{P,a}= -\frac{z-z(a)}{\prod_{q\in Q}(z-z(q))}dz^2. 
\]
At a point where $\eta_{P,a}$ has neither a zero nor a pole, there exists a local coordinate $w$ such that $\eta_{P,a}=dw^2$. So $|\eta_{P,a}|$ defines a flat metric at such a point. If $\eta_{P,a}$ has  order $k\ge -1$ at $z$, then this metric extends there, and makes the space  there locally a Euclidean cone 
with total angle $(k+2)\pi$. For $k=-2$, the metric does not extend and a punctured neighborhood will be metrically an infinite cylinder.
But $\eta_{P,a}$ gives more than a metric. If we regard it as a meromorphic function on the tangent bundle of $P$, then the tangent vectors on which  $\eta_{P,a}$ takes a real nonnegative value define a foliation on $P$ with singularities. 

In order that $\eta_{P,a}$ be a Jenkins-Strebel differential, all the leaves of this foliation save for a finite number must be compact. The theory asserts that this happens for exactly one choice of $a$. We  then write $a_P$ for $a$ and $\eta_P$ for $\eta_{P,a}$.  The compact leaves then encircle the point at infinity and have length $2\pi$ (when the length is measured via $\eta_P$),  and the noncompact leaves make up a tree $D_P $ connecting the three orbifold points of order $2$ with total length $\pi$. So the metric $|\eta_P|$ makes $P^\circ\ssm D_P$ isometric to the product of the unit circle and an open Euclidean half line.

When $a_P\notin Q$, the tree $D_P$ has three edges  connecting  $a_P$ with a point of $Q$, otherwise $D_P$ degenerates  in an obvious way into a graph with two edges. The form $\eta_P$ endows $P^\circ$ with a flat Euclidean structure which has singularities in $Q$ (with angular excess $-\pi$ for any $q\in Q$ and  $+\pi$ at $a$, albeit that the singularities at $q\in Q$ and $a_P$ cancel each other out when $a_P=q$.  This gives $P$ a {\em ribbon structure} with $D_P$ the associated ``ribbon tree''; see for example \cite{looijenga:JS}.  
The ribbon structure manifests itself only when $a_P\notin Q$ and then amounts to a cyclic ordering of $Q$.

Since the metrized ribbon tree $D_P$ is a complete invariant of  $P$ (see for example \cite{looijenga:JS}), the moduli space of complex orbifolds of type  $(0; 3,2,2,2)$ is thus identified with the space of triples of nonnegative real numbers with sum $\pi$ of which at most one is zero  and which are cyclically ordered when all are nonzero. By dividing these lengths by $\pi$, we can interpret such triples as the barycentric coordinates of a $2$-simplex; we thus have this moduli space identified with the standard $2$-simplex modulo its cyclic group of automorphisms (of order $3$) minus the image of the point
that represents the vertex set of the simplex. This space is homeomorphic to a sphere minus a point, the missing point corresponding to the case when (exactly) two points of $Q$ coalesce to a single point $q_0$; this is evidently a degenerate case. The associated Jenkins-Strebel differential  is then 
\begin{equation}\tag{degen}
\label{eqn:simple}
\eta_P=-dz^2/(z-z(q_0))(z-z(q_1)),
\end{equation}
where $q_1$ denotes the remaining point (we carry $P$ along in the notation, but $P$ should now be thought of as a degenerate orbifold) and the associated graph $D_P$ will be an arc connecting $q_0$ with $q_1$.

If $f:C\to P$ is a smooth $\Af_5$-cover (so that $f^{-1}q_\infty$ is a size $20$-orbit), then $f^*\eta_P$ defines on $f^{-1} P^\circ=C\ssm f^{-1}q_\infty$ a complete $\Af_5$-invariant flat Euclidean structure with singularities. The singular  locus of this metric is $f^{-1}a_P$, for the angular excess of 
the preimage of an order $2$ orbifold point  $\not=a_P$  will vanish because of the branching. When $a_P\notin Q$, the angular excess at 
any point of  $f^{-1}a_P$ is $+\pi$, otherwise it will be $2\pi$, because of the ramification. In fact, $f^*\eta_P$ is a Jenkins-Strebel differential associated to the pair 
$(C,f^{-1}q_\infty)$. It is characterized by the property that its double residue at each point of $f^{-1}q_\infty$ is $-9$. Its associated metrized ribbon graph is $f^{-1}D_P$ and has $f^{-1}a_P$ as its set of 
genuine vertices (i.e., vertices that are not of order $2$). A connected component of $C\ssm f^{-1}D_P$ triply covers the open disk $P\ssm D_P$ with ramification over $q_\infty$ and hence is itself an open disk.
When $a_P\notin Q$, the closure of such a component has the structure of a one point compactification of a metrical  product of a circle 
of circumference  $6\pi$ and a Euclidean half line $\Rb_{\ge 0}$. Its boundary is a $9$-gon contained  in $f^{-1}D_P$ which has as vertex set  its intersection with  $f^{-1}a_P$. Each of its nine edges contains a ramification point of $f$. So this gives $C$ a canonical $\Af_5$-invariant cell structure.

Let us see what happens when two points of $Q$ coalesce to a point $q_0$. We denote by $q_1$ the inert point so that now $Q_0=\{q_0,q_1\}$ and  
(in anticipation of the orbifold structure we have yet to define) we denote the projective line with the three points $q_\infty, q_0, q_1$ by $P_0$.
Then $\eta_{P_0}$ and $D_{P_0}$ have the simple form (\ref{eqn:simple}) displayed above.
The $\Af_5$-covers over such a degenerate orbifold $P_0$ are nodal curves such that the preimage of $q_0$ is the set  of nodes. 
Let $C_0$ be the normalization of an irreducible  component of the associated cover.  Theorem \ref{theorem:moduli1} tells us that  $C_0$ is a copy of $\Pb^1$ and has 
 stabilizer  conjugate to  $\Sf_3^{\ev}$ (Petersen graph), $\Af_4$ ($K_5$ graph) or  $\Af_5$ (irreducible case).  The  map $f_0: C_0\to P_0$ has ramification of order 
$3$  over $q_\infty$ and of order $2$ over $q_1$ as before, but the ramification over $q_0$ depends on the case.  

When $C$ is smooth, any point of $C$ lying over $Q$ has stabilizer of order $2$.  When two such points are forced to merge in a single branch point of $C_0$
then this point on $C_0$ has as its stabilizer the  subgroup of the generated by two elements of order $2$ in the stabilizer of $C_0$.
This will be of order $2$ in the Petersen case, $3$  in the  $K_5$  case and $5$ in the dodecahedral (irreducible) case. In other words, we may regard $q_0$ as an orbifold point for the cover $C_0\to P_0$ of order $2$, $3$ or $5$ respectively.

The pull-back $f_0^*\eta_{P_0}$  will  still be a Jenkins-Strebel differential, and its order of vanishing in a point of $f_0^{-1}q_0$ will be respectively 
$0$, $1$ and $3$. In the last two cases, the associate graph is just $f_0^{-1}D_{P_0}$, but in the  first case this 
graph is in fact empty:  each connected component is an infinite cylinder foliated by circles, one of these being $f_0^{-1}D_{P_0}$. We therefore put $D_{C_0}:=f_0^{-1}D_{P_0}$. So in the Petersen case, $D_{C_0}$ is a circle containing the three points of $f_0^{-1}q_0$ and the three points of 
$f_0^{-1}q_1$.  In the $K_5$ case resp.\   dodecahedral case,  each connected component of  $D_{C_0}$ is the $1$-skeleton of a regular tetrahedron resp.\ icosahedron   having $f_0^{-1}q_0$ as vertex set and $f_0^{-1}q_1$ as the set of midpoints of the edges.  These regular polyhedra are however punctured at  the barycenters of their faces and the piecewise flat metric makes every closed face  deprived from its barycenter   a half cylinder. In other words,  in each of the three cases every face---a solid equilateral triangle---is replaced by the product of its boundary and a Euclidean half line. (So from this perspective we end up with an icosahedral model rather than a dodecahedral model as depicted in Figure \ref{figure:Dodec1}.)

We thus obtain another combinatorial model of the Wiman-Edge pencil. It differs from the one we found in \S\ref{subsect:combmodel1}. For example, if $Q$ has cyclic symmetry of order $3$ (so that we can take our affine coordinate $z$ such that $z(Q)$ are the third roots of unity), then the graph considered in \S\ref{subsect:combmodel1} is the triangle spanned by $Q$, whereas $D_P$ 
consists of the three rays emenating from the barycenter of this triangle to its vertices.  But the underlying tesselation of $\Bc $ is the same as for the model based on affine geometry, because the loci of exceptional orbifolds has the same characterization. For example, when $ \ell_1=\ell_2$ means that $q_3$ is equidistant to $q_1$ and $q_2$. We then have a real structure on $P$ for which $\{q_1, q_2\}$ is a conjugate pair and 
$\eta_P$ will be defined over $\Rb$ relative to this structure. The situation is similar when the three points of $Q$ are collinear.

\subsection{A model based on hyperbolic geometry} 
We here endow each smooth member $C_t$ of the Wiman-Edge pencil with the unique hyperbolic metric compatible with its complex structure.
The $\Af_5$-action on $C_t$ by biholomorphic automorphisms then preserves this hyperbolic metric, giving its orbit space the structure of a hyperbolic orbifold. We therefore begin our discussion with considering such orbifolds. 

So let $P$ be an orbifold of type $(0; 3,2,2,2)$ in the holomorphic category, but endowed 
with the associated complete hyperbolic structure: so the total angle at the order $3$ point $q_\infty$ is $2\pi/3$ and the total angle at an order 
$2$ point $q\in Q$ is $\pi$.   Let $\delta$ be a geodesic arc in $P$ connecting two orbifold points of order $2$ which does not pass through 
the other orbifold points. If $f:C\to P$ is an $\Af_5$-cover, then $f^{-1}\delta$ consists of a collection of pairwise disjoint closed geodesics. 
These geodesics  decompose $C$ into hyperbolic genus  zero  surfaces-with-boundary.  Note that if we let the two points of $Q$ connected by $\delta$ coalesce to a single point $q_0$, then we obtain we have  a complete orbifold  hyperbolic structure on $P_0\ssm\{ q_0, q_\infty\}$ (which  induces a  complete hyperbolic structure on its orbifold covering). 

The distance between $\delta$ and the remaining point $q_1$ of $Q$ will be realized be a geodesic $\delta'$.  If $\delta'$ hits $\delta$ in its interior, 
then it will do so orthogonally. If $\delta'$ meets $\delta$ in an endpoint, then since point is an orbifold point of order $2$, the passage to a local orbifold cover shows that $\delta'$ meets $\delta$ in the opposite direction (so that both lie on a  complete geodesic).
Since $q_\infty$ is on orbifold point of order $3$,  $\delta'$ does not pass through $q_\infty$. So $U:=P\ssm\{\delta\cup\delta'\}$ is  a topological disk
having $q_\infty$ as an interior point.  This hyperbolic decomposition  of $P$ gives rise to a paving of $C$
into hyperbolic right-angled hexagons whose alternating sides have equal length, namely $2\ell(\delta)$ and $2\ell(\delta')$. To see this, note  that
each connected component $\tilde U$ of $f^{-1}U$ triply covers $U$ with  the unique point $\tilde q_\infty\in \tilde U$ that lies over  $q_\infty$ 
as the only point of ramification. In particular, the boundary $\partial \tilde U$ triply covers $\delta\cup\delta'$  and we then easily find that $\tilde U$  is indeed the interior of a hyperbolic 
right-angled hexagon of the asserted type. 

Noteworthy is the case when $P$ is invariant under an involution, as this  yields the Wiman curve. Then there exists an affine coordinate 
$z$ for $P\ssm \{q_\infty\}$ such that $z(Q)=\{-1,0,1\}$, the involution taking $z$ to $ -z$. We take here $\delta\cup\delta'=[-1,1]$ (which is 
indeed the union of two geodesics). So then our right-angled hexagons have all sides of equal length. The right angles are at points of  the orbit $f^{-1}(0)$ and the two orbits $f^{-1}(1)$  and $f^{-1}(-1)$ consists of midpoints of edges. When the involution lifts to $C$, then this lift and the covering transformations endow  $C$ with an $\Sf_5$-action, so that $C$ will be a copy of the Wiman curve.

If we succeed in shrinking $\delta$ to a point $q_0$, while taking along a $\delta'$ as above, then $\alpha'$ becomes a geodesic ray $\alpha'_0$ from $q_1$ to the cusp $q_0$. Now $f_0^{-1}\alpha'_0$ will be a tesselation of $C_0\ssm f_0^{-1}q_0$ into ideal triangles (which may be regarded as degenerate right angled hexagons). But we do not know whether we can do this in a uniform manner as we did for  piecewise flat structures
on $P$ using Jenkins-Strebel differentials.

\section{The Wiman curve via hyperbolic triangle groups}
\label{section:triangle}
The goal of this  section is to investigate the natural hyperbolic structure on the Wiman curve. 
We shall later establish that it is in fact a Shimura curve of indefinite quaternionic type.

\subsection{Spinor norm}
This subsection merely serves to recall the definition and a few simple properties related to 
the notion of spinor norm.  

Let $V$ be a finite-dimensional vector space over a field $F$ endowed with  a quadratic form $q: V\to F$. This means that $q(\lambda v)=\lambda^2 q(v)$ for $\lambda\in F$ and $v\in V$ and that the map $b:V\times V\to F$ given by 
\[
b(v,v'):=q(v+v')-q(v)-q(v')
\] is bilinear. Note that $b(v,v)=2q(v)$, so that we can recover $q$ from $b$ if ${\rm char}(F)\neq 2$. 
Given $u\in V$ with $q(u)\not=0$, then the orthogonal reflection $s_u:V\to V$ with respect to $u$   is the element of $\Orth (V)$ defined by 
\[
s_u(v):=v- q(u)^{-1}b(v,u)u.
\]  
It is clear that $s_u=s_{u'}$ if and only if  $u'=\lambda u$ for some $\lambda\in F^\times$ and so the image of $q(u)$ in 
$F^\times/(F^\times)^2$ is an invariant of $s_u$. The \emph{spinor norm} is a group homomorphism $\epsilon: \Orth (V)\to F^\times/(F^\times)^2$ characterized by the property that it assigns to $s_u$ the image of $q(u)$ in $F^\times/(F^\times)^2$. We denote its  kernel by $\Orth^\#(V)$.  When every element of $F$ is a square (which is for instance the case when $F=\Fb_2$ or $F=\Cb$), then  $\Orth (V)=\Orth^\#(V)$ and this notion is devoid of interest.

This is not so when $F=\Rb$, for we may then identify  $F^\times/(F^\times)^2$ with $\{\pm 1\}$. In that case we have the following alternate definition of the spinor norm. If  $q$ has signature $(p,n)$, then the 
negative definite  linear subspaces $W\subset V$ of dimension $n$ make up a contractible open subset $\Db (q)$ in the Grassmannian (it is the 
symmetric space of $\Orth (V)$) and hence the tautological $n$-plane bundle over $\Db (q)$ is trivial. So this bundle can be oriented. Any reflection $s_u:V\to V$ will leave invariant a negative definite  $n$-plane $N$.  It will act on $N$ as the identity resp.\ as a reflection (and hence be 
orientation-reversing in $N$) if $q(u)>0$ resp.\ $q(u)<0$.
Thus the spinor norm  of $g\in \Orth (V)$ is $1$ if and only if the action of  $g$ on this bundle is orientation preserving. The spinor norms attached to $q$ and $-q$ define a homomorphism $\Orth (V)\to \{\pm 1\}\times\{\pm 1\}$ whose kernel is the identity component of $\Orth(V)$; if $q$ is indefinite, then this map is also onto. Note that the product of these two spinor norms is the determinant $\det: \Orth(V)\to \{\pm1\}$. In particular, the
two spinor norms coincide on $\SO (V)$, and  the common value on  $g\in\SO (V)$  is $1$ if and only if it induces an orientation preserving diffeomorphism of $\Db(q)$.

\subsection{Triangle groups of compact hyperbolic type} For a triple $(p_1,p_2,p_3)$ with $p_i\in\{2,3,4, \dots, \infty\}$ (and here ordered as $p_1\ge p_2\ge p_3$), 
the {\em triangle group} $\Delta(p_1,p_2,p_3)$ is the Coxeter group defined by generators $s_1,s_2,s_3$ subject to the relations $s_i^2=(s_{i-1}s_{i+1})^{p_i}=1$, where we  let $i$ run  over $\Z/(3)$.  The Tits reflection representation (as described for example in Bourbaki) makes 
$\Delta(p_1,p_2,p_3)$ a subgroup of 
$\GL_3(\Rb)$, as follows: we define a symmetric bilinear form $b: \Rb^3\times  \Rb^3\to \Rb$ by $b(e_i, e_i)=1$ and $b(e_i, e_j):=-\cos (\pi/p_{i+j})$ for $i\not=j$; we then assign to $s_i$ the $b$-orthogonal reflection with respect to $e_i$:
\[
s_i(x)=x-2b(x, e_i)e_i.
\]
The reflection representation is actually  defined to be the dual of this representation: a fundamental chamber is  then given by the $\xi\in (\Rb^3)^\vee$ that are positive on
each $e_i$. When $b$ is nondegenerate, the passage to the dual is not necessary, 
as this representation is then self-dual and we can take as the fundamental chamber the set defined by $b(x, e_i)> 0$.
The index two subgroup $\Delta^+(p_1,p_2,p_3)\subset \Delta(p_1,p_2,p_3)$ of 
orientation-preserving elements is generated by
$g_i:=s_{i-1}s_{i+1}$ ($i\in\Z/3$) and has as a complete set of relations $g_i^{p_i}=1$ (for $p_i=\infty$ read this as the empty relation) and $g_1g_2g_3=1$.
\\

Now assume that we are in the compact hyperbolic  case: this means that each $p_i$ is finite and that $1/p_1 +1/p_2+1/p_3<1$.
Then $b$ is indeed nondegenerate with  signature $(2,1)$ 
and there is a unique connected component $C^+$ of the of set  $b(x,x)<0$ for which the  closure of the fundamental chamber is contained in $C^+\cup\{ 0\}$. So 
the projection of $C^+$ in the real projective plane $\Pb^2(\Rb)$ is a hyperbolic disk $\Db$ (which is in fact the symmetric space of $\Orth(b)$) and the closed fundamental chamber defines a solid geodesic triangle $\Pi$ in $\Db$. The classical Schwarz theory (which is here subsumed by the theory of reflection groups) tells us that $\Delta(p_1,p_2,p_3)$ 
acts faithfully and properly discretely on  $\mathbb{D}$ and has the compact $\Pi$ as a strict fundamental 
domain. So for any $i\in \Z/(3)$, the hyperbolic quadrangle 
$\Pi\cup s_i\Pi$ serves as a fundamental domain for $\Delta^+(p_1,p_2,p_3)$.

\subsection{The triangle group attached to the Wiman curve}
\label{subsection:triangle}
Let us now focus on the case of interest here, where $(p_1,p_2, p_3)=(6,4,2)$.  
So then 
\begin{align*}
b(e_2, e_3)&=-\cos (\pi/6)=-\tfrac{1}{2}\sqrt{3},\\
b(e_3, e_1)&=-\cos (\pi/4)=-\tfrac{1}{2}\sqrt{2},\\ 
b(e_1, e_2)&=-\cos (\pi/2)=0.
\end{align*}

This can be made part of system of root data in several ways.
We take $\alpha_1:=2e_1$ and $\alpha_2:=\sqrt{6}e_2$, $\alpha_3:=\sqrt{2}e_3$ so that
\[
b(\alpha_i, \alpha_j)=
\begin{pmatrix}
4 & 0 & -2\\
 0 & 6 & -3\\
 -2 & -3 & 2\\
\end{pmatrix}.
\]
This allows us to define  a  generalized root system for which the $\alpha_i$ are 
the simple roots  and  $\alpha_i{} ^\vee:=2b(\alpha_i, -)/b(\alpha_i, \alpha_i)$ is the  coroot corresponding to $\alpha_i$. Then  $\alpha_i{}^\vee(\alpha_j)$ 
(an entry of the Cartan matrix) is integral and $s_i(x)=x-2\alpha_i{}^\vee (x)\alpha_i$. So the \emph{root lattice} 
\[
L:=\Z \alpha_1+\Z\alpha_2+\Z\alpha_3
\]
is  invariant under $\Delta(6,4,2)$ and $b$ is an even, integral, symmetric bilinear form on $L$. In particular, we have an associated quadratic form $q: L\to \Zb$, $q(v)=\tfrac{1}{2}b(v,v)$. (We could have taken $\alpha_1$ to be $e_1$ instead of $2e_1$, but then $b$ would not be even and so $q$ would not be $\Zb$-valued.) The  discriminant of $b$ is easily computed to be $-12$. 

Let us observe  that if we write $x=\sum_ix_i\alpha_i$, then 
\begin{multline*}
2q(x)=b(x,x)=4x_1^2+6x_2^2+2x_3^2-4x_1x_3-6x_2x_3=(2x_1-x_3)^2-3x_2^2+(x_3-3x_2)^2,\\
\end{multline*}
If we reduce this  modulo $3$, we get $(2x_1-x_3)^2 +x_3^2$  and this  quadric has no solution in $\Pb^2(\Fb_3)$.
So the conic $K_q\subset \Pb^2$ defined by $q$ has no rational point.

We thus get  $\Delta(6,4,2)\hookrightarrow \Orth(q)$ and hence 
$\Delta^+(6,4,2)\hookrightarrow \SO(q)$. Vinberg's theory of hyperbolic reflection groups shows that $\Delta(6,4,2)$  the $\Orth(q)$-stabilizer of the cone $C^+$, that is, to $\Orth^\#(q)$. 
So $\Delta^+(6,4,2)= \SO^\#(q)$.  

Let us consider the reduction mod  $5$ of $q$,  
\[
q_{\Fb_5}: L_{\Fb_5}\to \Fb_5.
\]
Since $b$ has discriminant $-12$, we know that $b_{\Fb_5}$ is nondegenerate, but as $-12$ is not a square mod $5$, it  is not equivalent  to the standard diagonal form modulo $5$.   The 
identity 
\[
q(-\alpha_1+2\alpha_2+2\alpha_3)=10
\]
shows that $K_q(\Fb_5)$ is nonempty, and so we can identify $K_q(\Fb_5)$ with $\Pb^1(\Fb_5)$.
We thus obtain a homomorphism  
\[
\Orth(q)\to \Aut (K_b(\Fb_5))\cong \PGL_2(\F_5).
\] 
It is known that $\PGL_2(\F_5)\cong \Sf_5$.

\begin{lemma}\label{lemma:homom3}
The natural homomorphism $\SO^\#(q)\to\Aut (K_b(\Fb_5))\cong \Sf_5$ is surjective with kernel the principal level $5$ subgroup  of
$\SO^\#(q)$ (denoted here by $\G_5$). The group $\G_5$  is torsion free so that the image of each $g_i$ in $\PGL_2(\F_5)$ has the same order as $g_i$.
\end{lemma}
\begin{proof}
Let us first show that the kernel is as asserted. Suppose $g\in \SO^\#(q)$ acts as the identity on $K_b(\Fb_5)$.  Let  $\overline\Fb_5$ be an algebraic closure of $\Fb_5$. An automorphism of 
$\Pb^1(\overline\Fb_5)$ which is the identity on $\Pb^1(\Fb_5)$ must be the identity (as it fixes three distinct points). So $g$ acts as the identity on
$K_b(\overline\Fb_5)$. Since we have a natural  identification of  $\Sym^2(K_b(\overline\Fb_5))$ with $\check\Pb^2(\overline\Fb_5)$ (a line in $\Pb^2(\overline\Fb_5)$ meets $K_b(\overline\Fb_5)$ in a degree 2 divisor), it follows that
$g$ acts as the identity on $\check\Pb^2(\overline\Fb_5)$ and hence also as the identity on $\Pb(L_{\Fb_5})$.
So $g\in \SO^\#(q)$ will act on $L_{\Fb_5}$ as scalar multiplication, say by $\lambda\in \Fb_5^\times$. Since $\det (g)=1$, we must have 
$\lambda^3=1$. Since, $\Fb_5^\times$ is cyclic of order $4$, it follows that $\lambda=1$.

For the remaining part we use Serre's observation that if an element of $\GL(n, \Z)$ has finite order, then  for any $\ell\ge 3$, 
its  mod  $\ell$ reduction  has the same order. We thus find that $g_1,g_2$ define in $\Sf_5$ elements of order $6$ and  $4$ respectively such that $g_1g_2$ has order $2$.  This does not yet imply that $g_1$ and $g_2$ generate $\Sf_5$: we must exclude the possibility that they generate a subgroup conjugate to  
$\Sf_3\times\Sf_2$). So it suffices to find an element in $\Delta^+(6,4,2)$  whose image in $\Aut (K_b(\Fb_5))$ has order $5$, or in view of the preceding,
whose image in $\SO^\#(q_{\Fb_5})$ has order $5$. One checks that $(s_1s_2s_3)^2\in \Delta^+(6,4,2)$ has this property.
\end{proof}

It follows that $\G_5$ acts freely on $\Db$, so that $\G_5\bs \Db$ is a compact orientable hyperbolic surface. It comes with a faithful 
action of $\SO^\#(L_{\Fb_5})\cong \Sf_5$ and a Riemann-Hurwitz count shows that it has genus $6$. We conclude:

\begin{theorem}\label{theorem:wiman:spinor3}
If we endow $\Db$ with an orientation such that $\G_5\bs \Db$ becomes a compact connected Riemann surface of genus $6$ with a faithful action of $\SO^\#(L_{\Fb_5})\cong \Sf_5$, then $\G_5\bs \Db$ is isomorphic to the Wiman curve. 
\end{theorem}

\begin{remark}[{\bf The antiholomorphic involution}]
   
The orientation reversing elements of $\Delta(6,4,2)$ are  reflections. There are three conjugacy classes of these, represented by $s_1,s_2, s_3$. That these are indeed distinct is a consequence of the evenness of the arguments of $\Delta(6,4,2)$,  for this implies that we have a surjection $\Delta(6,4,2)\to (\Z/2)^3$ which sends $s_i$ to the $i$th generator. Each $s_i$ determines a 
$\Sf_5$-conjugacy class of antiholomorphic involutions of $C$. We can interpret this by saying that there are exactly $3$ isomorphism types of nonsingular  projective curves of genus $6$ with $\Sf_5$-action  that are defined over $\Rb$.
\end{remark}

\section{Modular interpretation of the Wiman curve}
\label{section:modular}

The goal of this section is to prove Theorem \ref{theorem:wiman:modular4}. As mentioned in the introduction we prove in fact somewhat more precise result (Theorem \ref{theorem:wiman:modular4a}).  We continue with the terminology used in \S\ref{subsection:triangle}.

\subsection{Hodge structures parametrized by the Wiman curve}
\label{subsect:HS}
We fix an orientation on $L_{\Rb}$. As we will explain, this  turns $\Db$ into symmetric domain for $\SO^\#(L_\Rb)$ which parametrizes weight zero Hodge structures on $L$.  

We first observe that an element of  $\Db$ is uniquely represented by a vector $f$  in the cone $C^+$ (defined in Subsection \ref{subsection:triangle}
with  $q(f)=-1$. We shall use this to identify $\Db$ with a connected component  of the hyperboloid in $L_\Rb$ defined by this identity.
Note that for $f\in \Db$ the orthogonal complement of $f$  is positive definite.   So if $(e', e'')$ is an orthogonal \emph{oriented basis} of this orthogonal complement then the complex structure 
on the tangent space of $\Db$ at $f$ is given by the transformation which takes $e'$ to $e''$ and $e''$ to $-e'$. Note that 
$e:=e'+\sqrt{-1}e''$ is the $-\sqrt{-1}$ eigenspace for this complex structure and that the map \[T_f\Db\to \Cb\otimes_\Rb T_f\Db\to \Cb\otimes_\Rb T_f\Db/\Cb e\] is then an isomorphism of complex lines. This shows that the map $\Db\to \Pb(L_\Cb)$ which assigns to $f$ the complex-linear span of $e$ is
holomorphic.  Note that $e$ has the property that $b_\C(e,e)=0$ and $b_\C(e, \overline e)>0$. This defines a weight zero Hodge structure on $L$ which only depends on $f$: 
$L^{1,-1}_f$, $L^{-1,1}_f$, $L^{0,0}_f$  is spanned by $e$, $\bar e$, $f$ respectively. This Hodge structure is polarized by $-b$.
The locus in $ \Pb(L_\C)$ defined $b_\C(e,e)=0$ and $b_\C(e, \overline e)>0$ has two connected components
and the map just defined $f\in \Db\mapsto [L^{1,-1}_f] \in  \Pb(L_\C)$ identifies $\Db$ with one of these in the holomorphic category. In particular, $\Db$ becomes a symmetric domain for $\SO^\#(L_\Rb)$.

Thus the Wiman curve parametrizes Hodge structures of this type. There exists an isometric embedding of $L$ in a $K3$-lattice (this is an even unimodular lattice of signature $(3, 19)$) and any two such  differ by an orthogonal transformation of the $K3$-lattice. This enables us to let the Wiman curve parametrize  $K3$-surfaces (in fact, Kummer surfaces), but we prefer to set up things in a more canonical fashion which for instance takes explicitly into account the $\Af_5$-symmetry. This involves  a Kuga-Satake construction in the spirit of Deligne \cite{deligne}.

\subsection{The Clifford algebra associated to our triangle group}

The Clifford algebra associated to the quadratic lattice $(L,q)$  is by definition the quotient $C(q)$ of the tensor algebra of $L$ by the 
$2$-sided ideal generated by the even tensors of the form $v\otimes v-q(v)$, $v\in L$. This algebra clearly comes with an action of $\Orth(q)$. 
The anti-involution which reverses the order of a pure tensor,
$a=v_1\otimes \cdots\otimes v_r\mapsto a^*:=v_r\otimes \cdots\otimes v_1$, descends to an anti-involution in $C(q)$. In view of the fact that
$aa^*=q(v_1)\cdots q(v_r)$, the map $a\in C(q)\to aa^*\in\Zb$ is a quadratic function  which extends $q$. We therefore continue to denote it by $q$;
it is known as the \emph{norm} on $C(q)$. So if $q(a)\not=0$, then $a$ is invertible in $C(q)_\Qb$ with $2$-sided inverse $q(a)^{-1} a^*$.

For $r=0,1,2,\dots$, the images of the tensors of order 
$\le r$  define a filtration of $C(q)$ as an algebra whose associated graded algebra is the exterior algebra of $L$. In particular, $C(q)$ is a free $\Zb$-module 
of rank $2^3=8$. 
The splitting into even and odd tensors defines a corresponding splitting $C(q)=C_+(q)\oplus C_-(q)$, with $C_+(q)$ being a subalgebra $C(q)$ of rank $4$. 

\begin{lemma}\label{lemma:}
The algebra $C_+(q)$ is a division algebra of indefinite quaternion type. In other words, for every $a\in C_+(q)\ssm\{0\}$, $q(a)\not=0$ and $C_+(q)_\Rb$ 
is isomorphic  to $\End_\Rb(\Rb^2)$
\end{lemma}
\begin{proof}
The fact  that $q$ does not represent zero implies that $C_+(q)$  is a division algebra. For this we must show that $q(a)\not=0$  when $a\in C_+(q)\ssm \{0\}$. If $\beta\in L_\Qb$ is nonzero, then extend $\beta$ to an orthogonal basis $\beta=\beta_1, \beta_2, \beta_3$ of $L_\Qb$. Then it is clear that right multiplication with $\beta$ takes the basis 
$1, \beta_1\beta_2, \beta_2\beta_3, \beta_3\beta_1$ of $C_-(V)_\Qb$ to a basis of $C_-(V)_\Qb$.  In particular,  $a\beta\not=0$. It follows that
$a L_\Qb$ is a $3$-dimensional subspace of  the $4$-dimensional subspace $C_-(V)_\Qb$. So $\dim (L_\Qb\cap a L_\Qb)\ge 2$. Now choose $\beta\in  L_\Qb$ such that $a\beta$ is a nonzero element of $L_\Qb$. Then $0\not= q(a\beta)=q(a)q(\beta)$ and so $q(a)\not=0$. 

The assertion that $C_+(q)_\Rb\cong \End_\Rb(\Rb^2)$ follows from the fact that $q_\Rb$ has hyperbolic signature.
\end{proof}

Note that if $v,x\in L$, then  in $C(q)$ we have 
\[
vxv=v(-vx+b(x,v))=-q(v)x+b(x,v)v,
\]
which in case $q(v)\not=0$, is just $-q(v)$ times the orthogonal reflection $s_v$ in $v$. 

We define the group  $\spin (q_\Rb)$ as the group of units $u$ of $C_+(q)_\Rb$ with $q(u)=1$ and with the property that conjugation with $u$ preserves $L_\Rb$. This conjugation acts in $C(q)_\Rb$ as an orthogonal transformation and the resulting
group homomorphism $\spin (q_\Rb)\to\Orth (q_\Rb)$ has image $\SO^\# (q_\Rb)$ and kernel $\{\pm 1\}$. The group $\spin (q_\Rb)$ acts (faithfully) on $C_+(q)_\Rb$ by left multiplication  (this is the spinor representation) and also by right multiplication  after inversion. The two actions clearly commute and as we noted, the diagonal action (so given by conjugation) factors through  $\SO^\# (q_\Rb)$.

We write $\spin (q)$ for the preimage of  $\SO^\#(q)$ in  $\spin (q_\Rb)$ so that we have an exact sequence
\[
1\to \{\pm1\}\to \spin (q)\to \SO^\#(q)\to 1.
\]
We can use  the root basis $\alpha_1, \alpha_2,\alpha_3$ of $L$ to obtain presentations of $C(q)$ and $C_+(q)$. A presentation for $C(q)$
has these basis elements as generators and is such that  $\alpha_i\alpha_j+\alpha_j\alpha_i=b(\alpha_i, \alpha_j)$. So
\begin{center}
$\alpha_1\alpha_1=2,\; \alpha_2\alpha_2=3,\; \alpha_3\alpha_3=1$,\\
$\alpha_1\alpha_2+\alpha_2\alpha_1=0, \; \alpha_2\alpha_3+\alpha_3\alpha_2=-3, \; \alpha_3\alpha_1+\alpha_1\alpha_3=-2$.
\end{center}
In particular, $\alpha_1^{-1}=\tfrac{1}{2}\alpha_1$, $\alpha_2^{-1}=\tfrac{1}{3}\alpha_2$ and  $\alpha_3^{-1}=\alpha_3$. For $x\in L$,  
\[
\alpha_1x\alpha_1=-2s_1(x), \; \alpha_2x\alpha_2=-3s_2(x), \; \alpha_3x\alpha_3=-s_3(x).
\]

Note that $a_1:=\alpha_2\alpha_3$, $a_2:=\alpha_3\alpha_1$ and $a_3:=\alpha_1\alpha_2$ generate $C_+(q)$. In fact, $a_1$ and  $a_2$ will do, since $a_1a_2=-a_3$.  We also have:
\[
a_1^2+3a_1 +3=0,\; a_2^2+2a_2 +2=0,\; a_3^2+6=0
\]
Note that for $x\in L$, $a_ixa_i^{-1}=g_i(x)$, so that $\tilde g_i:=a_i/\sqrt{q(a_i)}\in \spin (q_\Rb)$ is a lift of $g_i$ and hence lies in $\spin (q)$. This implies that $\tilde g_1^{6}$, $\tilde g_2^{4}$ and $\tilde g_3^{2}$ are all equal to $-1$. One checks that in addition, $\tilde g_1\tilde g_2\tilde g_3=-1$. Since the $g_i$ generate $\Delta^+(6,4,2)=\SO^\#(q)$, the  $\tilde g_i$ will generate $\spin (q)$.

We identified  $\G_5\subset \SO^\#(q)$ as the principal level $5$ subgroup of $\SO^\#(q)$. So its preimage $\hat\Gamma_5$ in $\spin (q)$, which 
is central a extension of $\G_5$ by $\{\pm 1\}$,  is the subgroup of $\spin (q)$ that will act as the identity on the mod $5$ reduction of $L$. Hence
it will act trivially on the mod $5$ reduction of $C_+(q)$. The converse holds  up to sign, in the sense that if $g\in \SO^\#(q)$ acts as the identity on
$C_+(q)\otimes\Fb_5$, then $g\in \{ \pm 1\}.\G_5$. To see this, note that we may identify $C_+(q)/\Zb$ with $\wedge^2 L$. The assertion is then a consequence of the fact that if $g$ acts trivially on $(\wedge^2 L)_{\Fb_5}\cong \wedge^2_{\Fb_5} L_{\Fb_5}$, then $g$ acts as $\pm 1$ on
$L_{\Fb_5}$.
 
\subsection{Abelian surfaces parametrized by the Wiman curve}

We noted that every  $f\in \Db$ defines a weight zero Hodge structure  $L_f$ on $L$ polarized by $-b$. We extend this to a 
Hodge structure on the tensor algebra of  $L$. The $2$-sided ideal we divide out by to get $C(q)$ is a subHodge structure, and thus
$C(q)$ receives a Hodge structure. The decomposition  $C(q)=C_+(q)\oplus C_-(q)$  is then one into 
Hodge structures. This Hodge structure  has bidegrees $(1,-1)$, $(0,0)$ and $(-1,0)$ only. To be concrete, let
$e=e'+\sqrt{-1}e''\in L_\Cb$ be as in Subsection \ref{subsect:HS}, so that   $e$ spans $L^{1,-1}_f$ and $b(e,\bar e)>0$. Then  $f$ spans $L^{0, 0}_f$, $\bar e$ spans $L^{-1, 1}_f$ and  $(ef, e\bar e, \bar e e,  \bar e f)$ is a basis of $C_+(q)_\Cb$ (note that $e\bar e+\bar e e=1$) with bidegrees $(1,-1)$, $(0,0)$, $(0,0)$ and $(-1,1)$ respectively.
Now $j_f:=e'e''\in C_+(q)_\Rb$ satisfies the identity $\half e\bar e =1-\sqrt{-1}j_f$ and hence only depends on $L^{1,-1}_f$, is of bidegree $(0,0)$ and has the property that 
$j_fj_f=e'e''.-e''e'=-1$.  So  \emph{right} multiplication by $j_f$ defines a complex structure $J_f$ on $C_+(q)_\Rb$.  It takes $e'$ to $e''$ and $e''$ to $-e'$ and so $J_b$   lifts  the complex structure in $T_f\Db$ under the covering projection $\spin (q_\Rb)\to \SO^\# (q_\Rb)$.

 It is clear that this complex structure is preserved by the action of  $C_+(q)_\Rb$  on the left of $C_+(q)_\Rb$, so that this turns  $(C_+(q)_\Rb, J_f)$ into a complex representation of $C_+(q_\Cb)$ of dimension $2$. 

\begin{lemma}\label{lemma:cxstructure}
Every  complex structure on $C_+(q)_\Rb$ which leaves the left $C_+(q)_\Rb$-module structure invariant is equal to $J_f$ or $-J_f$ for some $f\in\Db$.
\end{lemma}
\begin{proof}
This is perhaps best seen by choosing an $\Rb$-algebra isomorphism $C_+(q)_\Rb$ with $\End_\Rb (W)$, where $W$ is a real vector space of dimension $2$.  Via the natural identification  $\End_\Rb (W)\cong W\otimes_{\Rb} W^\vee$, this identifies $\spin(q_\Rb)$ with $\SL(W)\times \SL (W^\vee)$. A  complex structure on $\End_\Rb (W)$ which preserves the  left $\End_\Rb (W)$-module structure amounts to a complex structure on $W^\vee$. The group $\SL(W^\vee)$ permutes these complex structures and two are in the same orbit if and only if they define the same orientation.
\end{proof}

We note that the generator $e$ of $L^{1,-1}_f$ satisfies
\[
J_f(e)=(e'+\sqrt{-1}e'')e'e''=e''-\sqrt{-1}e'=-\sqrt{-1}e, 
\]
so that $e$ lies in the  $-\sqrt{-1}$-eigenspace of $J_f$. It is in fact straightforward to verify that the $-\sqrt{-1}$-eigenspace of  $J_f$  in $C_+(q)_\Cb$ equals $C_-(q)_\Cb L^{1,-1}_f\subset C_+(q)_\Cb$. This implies that $C_-(q)_\Cb L^{1,-1}_f$ depends holomorphically on $f$. 
Hence so does the   complex torus
\[
A_f:= C_+(q)\bs C_+(q)_\Cb/C_-(q)_\Cb L^{1,-1}_f.
\]
Observe that its  first homology group $H_1(A_f)$ is naturally identified with the lattice underlying $C_+(q)$ (with its weight $-1$ Hodge structure); we may of course also think of $A_f$ as $C_+(q)\bs C_+(q)_\Rb$, where $C_+(q)_\Rb$ has the complex structure defined by $J_f$. 

Consider the map
\[
h: H_1(A_f)\times H_1(A_f)=C_+(q)^2\to C_+(q), \quad (a,b) \mapsto ab^*.
\] 
This can be regarded as a  $C_+(q)$-valued hermitian form on $H_1(A_f)$ with respect to the anti-involution $a\mapsto a^*$. The elements $h(a,b)-h(b,a)$ clearly lie in the 
$(-1)$-eigenspace of this involution in  $C_+(q)_\Qb$ and in fact span that space over $\Qb$ (for $h(a,1)-h(1,a)=a-a^*$). This is a subspace of $C_+(q)_\Qb$ of dimension $3$ which supplements $\Qb$.  If we identify $C_+(q)$ with $H_1(A_f)$ as above, then after dualization  we get the map
\[
E: C_+(q)^\vee\to H^1(A_f)\otimes H^1(A_f)\xrightarrow{\cup} H^2(A_f)
\] 
The preceding implies that $E$ has rank $3$.

\begin{lemma}\label{lemma:ns}
The left action of  $C_+(q)$ on itself yields an embedding of algebras $C_+(q)\hookrightarrow \End (A_f)$
and the image $E$ lies in  the Neron-Severi group of $A_f$.  
\end{lemma}
\begin{proof}
The first part of this assertion is clear. For the second statement, it suffices to show that  
the image of $E$ lies in $H^{1,1}(A_f)$, or equivalently,  that $h$ is  zero on $C_-(q)_\Cb L^{1,-1}_f\times C_-(q)_\Cb L^{1,-1}_f$. But the vector $e$ which spans $L^{1,-1}_f$ is isotropic and so $ee=0$ in  $C(q)_\Cb$. It follows that $h(xe, ye)=xeey^*=0$ for all $x, y\in  C(q_\Cb)$.
\end{proof}

We thus obtain a family $\As_{\Db}$ over $\Db$ with an action of $\spin (q)$.  Note that the action of $\spin (q)$ on $\Db$ factors through  $\SO^\#(q)=\spin(q)/\{\pm1\}$ and that $-1\in\spin (q)$ acts in each $A_f$ as the natural involution. With the help of Lemma \ref{lemma:cxstructure} we then find:

\begin{theorem}\label{theorem:wiman:modular4a}
The orbit space of this action by $\hat \G_5$ yields a moduli stack of Kummer surfaces $\{\pm 1\}\bs\As_{C_0}\to C_0$. To be precise, it classifies abelian surfaces $A$  endowed with an isomorphism $\End(A)_{\Fb_5}\cong C_+(q_{\Fb_5})$ which is the mod 5 reduction of  an algebra isomorphism $\End(A)\cong C_+(q)$ and for which  $H_1(A; \Zb)$ is a principal $\End(A)$-module.
This affords $C_0$ with the structure of a Shimura curve of indefinite quaternionic type.
\end{theorem}

\bigskip{\noindent
Dept. of Mathematics, University of Chicago\\
E-mail: farb@math.uchicago.edu\\
\\
Yau Mathematical Sciences Center, Tsinghua University, Beijing (China), and Utrecht University\\
E-mail: e.j.n.looijenga@uu.nl


\begin{thebibliography}{40}
\small

\bibitem{BR}
J.~Bertin and M.~Romagny, \textit{Champs de Hurwitz}, 
M\'em.\ Soc.\ Math.\ Fr.\ \textbf{125-126},  2011.


\bibitem{Cheltsov} 
I.~Cheltsov, C.~Shramov, \textit{Cremona groups and the icosahedron}. Monographs and Research Notes in Mathematics. CRC Press, Boca Raton, FL, 2016.

\bibitem{CKS} I.~Cheltsov, A.~Kuznetsov, C.~Shramov, \textit{Coble fourfold, $S_6$-invariant quartic threefolds, and Wiman-Edge sextics,}
\texttt{math.AG arXiv:1712.08906}.


\bibitem{deligne}
P.~Deligne,
\textit{La conjecture de Weil pour les surfaces $K3$,} Invent.\ Math. \textbf{15} (1972), 206--226.

\bibitem{DFL}
I.~Dolgachev, B. Farb and E. Looijenga, 
\textit{Geometry of the Wiman Pencil, I: algebro-geometric aspects}, 
Eur.\ J.\ Math.\  \textbf{4} (2018), 879--930  (see also \url{https://arxiv.org/abs/1709.09256}).

\bibitem{Edge} 
W.~Edge, 
\textit{A pencil of four-nodal plane sextics}, Proc.\ Cambridge Philos.\ Soc.\ \textbf{89} (1981), 413--421.

\bibitem{FL}
B. Farb and E. Looijenga, Arithmeticity of the monodromy of the Wiman-Edge pencil, preprint, Nov.\ 2019 
(see also \url{https://arxiv.org/abs/1911.01210}).

\bibitem{larcher} 
H.~Larcher,
\textit{The cusp amplitudes of the congruence subgroups of the classical modular group,} Illinois J.\ Math.\ \textbf{26} (1982),  164--172.

\bibitem{looijenga:JS}
E.~Looijenga, 
\textit{Cellular decompositions of compactified moduli spaces of pointed curves,} In: The moduli space of curves (Texel Island, 1994), 369--400, 
Progr. Math. \textbf{29}, Birkh\"auser Boston, Boston, MA, 1995. 

\bibitem{Sc}
P. Scott, The geometries of 3-manifolds, \textit{Bull. London. Math. Soc.} \textbf{15} (1983), 401--487.

\bibitem{Wiman} A.~Wiman, \textit{Ueber die algebraische Curven von den Geschlecht $p=4,5$ and $6$ welche eindeutige Transformationen in sich besitzen}. Bihang till Kongl. Svenska 
Vetenskjapsakad. Handlingar, {\bf  21} (1895), Afd. 1, No 3, 2--41.

\bibitem{Za}
A. Zamora, Some remarks on the Wiman-Edge pencil, preprint.

\end{thebibliography}
\end{document}